\documentclass[a4paper]{report}%[12pt]

\usepackage[ansinew]{inputenc}
\usepackage{amsthm,latexsym,amssymb,amsfonts,amsmath}
\usepackage{graphics,psfrag}
\usepackage{graphicx}
\usepackage[german,english]{babel}
\usepackage{hyperref}

\graphicspath{{pics/}}

\newcommand{\Datum}{21. Oktober 2009}

\newcommand{\R}{\mathbb{R}}
\newcommand{\dR}{r}%{\delta}
\newcommand{\ueber}[2]{\genfrac(){0pt}{}{#1}{#2}}
\newcommand{\ex}{\ensuremath{\operatorname{ex}}}
\newcommand{\conv}{\ensuremath{\operatorname{conv}}}
\newcommand{\id}{\ensuremath{\operatorname{id}}}
\newcommand{\sd}{\ensuremath{\operatorname{sd}}}

\newtheorem{thm}{Theorem}[section]
\newtheorem{cor}[thm]{Corollary}
\newtheorem{lem}[thm]{Lemma}
\newtheorem{prop}[thm]{Proposition}
\newtheorem{prob}[thm]{Problem}

\newenvironment{que}{\begin{list}{}{\addtolength{\leftmargin}{-2.5em}}\item[]\it}{\end{list}}

\theoremstyle{definition}  % Ab hier haben die Umgebungen ein etwas anderes Layout (z.B. nicht kursiv)
\newtheorem{Def}[thm]{Definition}
\newtheorem{conj}[thm]{Conjecture}
\newtheorem{rem}[thm]{Remark}

\begin{document}
\begin{titlepage}
 \Large 
 \noindent Freie Universit\"at Berlin\\
 Fachbereich Mathematik und Informatik\\
 \vspace{5cm}
 \begin{center}
  Diplomarbeit\\
  \vspace{1em}
  {\huge On the Complexity of\\ Embeddable Simplicial Complexes\\}
  \vspace{1em}
  Anna Gundert\\
  \Datum\\
  \vspace{5cm}
  Betreut von Prof.~Dr.~G\"unter M.~Ziegler\\
  (Technische Universit\"at Berlin)
 \end{center}
\end{titlepage}

\hspace{1em}\thispagestyle{empty}
\newpage
\thispagestyle{empty}
\section*{Danksagung}
F\"ur die Bereitstellung des Themas, viele hilfreiche Hinweise und die M\"oglichkeit, Anschlu\ss{} in seiner Arbeitsgruppe zu finden, m\"ochte ich mich bei dem Betreuer dieser Arbeit Prof.~Dr.~G\"unter M.~Ziegler bedanken.\\
Von der Arbeitsgruppe \glqq Diskrete Geometrie\grqq~der TU Berlin m\"ochte ich insbesondere Ronald Wotzlaw und Raman Sanyal f\"ur ihre stetige Hilfsbereitschaft Dank aussprechen. Mein weiterer Dank gilt Mathias Schacht f\"ur das Beantworten vieler Fragen zur extremalen Hypergraphentheorie.\\
Desweiteren danke ich Frau Monika Seid und Frau Sandra Breiter-Staufenbiel, die in gewissem Sinne die Voraussetzungen f\"ur das Verfassen dieser Arbeit geschaffen haben,  sowie meinen Eltern f\"ur ihre immer bedingungslose Unter\-st\"utzung.\\
Zu guter Letzt m\"ochte ich Frederik von Heymann f\"ur das gewissenhafte Lesen meiner Arbeit und st\"andigen Beistand danken.
\vspace{9.5cm}

\section*{Eidesstattliche Erkl\"arung}
Hiermit versichere ich, die vorliegende Arbeit selbst\"andig und unter ausschlie\ss{}\-licher Verwendung der angegebenen Literatur und Hilfsmittel erstellt zu haben.\\
Diese Arbeit wurde bisher in gleicher oder \"ahnlicher Form keiner anderen Pr\"u\-fungs\-kommission vorgelegt und auch nicht ver\"offentlicht.

\vspace{4em}
\noindent Berlin, den \Datum \hspace{14em} Anna Gundert
\tableofcontents
\chapter*{Abstract}
\addcontentsline{toc}{chapter}{Abstract}
%It is easy to prove that any $d$-dimensional simplicial complex embeds into $\R^{2d+1}$; so the maximal number of $d$-simplices one can get when embedding into $\R^{2d+1}$ is ${n \choose {d+1}} = \Theta(n^{d+1})$ where $n$ is the number of vertices.

%For the case $d=2$ this yields $\Theta(n^3)$ for embeddability into $\R^5$. One can also show that a 2-complex which embeds into $\R^3$ can have at most $n(n-3) = \Theta(n^2)$ triangles. What happens in $R^4$ is an open question. 
%more general

This thesis addresses the question of the maximal number of \mbox{$d$-simplices} for a simplicial complex which is embeddable into $\R^{\dR}$ for some $d \leq \dR \leq 2d$.

A lower bound of $f_d(C_{\dR + 1}(n)) = \Omega(n^{\lceil\frac{\dR}{2}\rceil})$, which might even be sharp, is given by the cyclic polytopes.
To find an upper bound for the case $\dR=2d$ we look for forbidden subcomplexes. A generalization of the theorem of van Kampen and Flores yields those. Then the problem can be tackled with the methods of extremal hypergraph theory, which gives an upper bound of $O(n^{d+1-\frac{1}{3^d}})$.
% $O(n^s), k<s<k+1$.

We also consider whether these bounds can be improved by simple means.
%We also show that using the same methods these bounds can in some sense not be improved.
%
\chapter*{Introduction}
\addcontentsline{toc}{chapter}{Introduction}
Questions on embeddability of simplicial complexes into Euclidean space are not only connected to topology but also to combinatorics. %relevant, associate link relate
Sarkaria's color\-ing/embedding theorem (Theorem 5.8.2 \cite{Matousek.2003}), e.g., links the non-embed\-dability of a complex $K$ into Euclidean space of certain dimensions with the chromatic number of a graph associated with this complex.

The most well-studied such embeddability question is that of planar graphs, which can be considered as $1$-dimensional simplicial complexes that allow an embedding into $\R^2$.
Analogous questions for higher dimensions are far less understood and seem to get more complicated. Several problems that are tractable for the case of planar graphs, are far more complex and partly open in higher dimensions. Here are a few examples:
%that have long been solved for the case of planar graphs remain to be clarified in higher dimensions.

While every graph that embeds into the plane also has a straight-line embedding,
there are examples of simplicial complexes of higher dimension that admit a topological embedding into some $\R^{\dR}$, but not a linear (or piecewise linear) embedding into this space (e.g., \cite{Brehm.1983}, \cite{Bokowski.2000}, \cite{Schewe.2006}, \cite[p.858]{Matousek.2008}). On the other hand, the planar case is not the only one in which different embeddability properties agree: It is, e.g., known that every $d$-dimensional simplicial complex that embeds topologically into $\R^{\dR}$ for some $\dR$ with $\dR-d \geq 3$ is also piecewise linearly embeddable into $\R^{\dR}$ \cite{Bryant.1972}.
% not every simplicial complex that admits a topological embedding into some $\R^{\dR}$ also admits a linear (or PL) embedding into this space.

Also the algorithmic complexity of deciding whether a given $d$-dimensional simplicial complex embeds into some $\R^{\dR}$ is not solved for all pairs $(d,\dR)$. For planarity testing polynomial algorithms have been developed. Some of the higher-dimensional cases also have polynomial complexity, in others the question is known to be NP-hard, or even undecidable \cite{Matousek.2008}.
%In some of the higher-dimensional cases the question is known to be NP-hard, in others to be undecidable (\cite{Matousek.2008}).

For the graph case it is well-known that the graphs $K_5$ and $K_{3,3}$ characterize non-planarity and are minimal non-planar graphs, i.e., all of their subgraphs are planar. The question of minimal non-embeddable complexes for higher dimensions is more complicated. While classes of such complexes are known (e.g., \cite{Gruenbaum.1969}, \cite{Zaks.1969b}, \cite{Sarkaria.1991}, \cite{Schild.1993}), it is also known that for $d\geq2$ and $\dR=2d$ no characterization via a finite set of minimal non-embeddable complexes is possible (\cite{Ummel.1973}, \cite{Zaks.1969a}).

% 
% If $K$ is a simplicial complex of dimension $d$, what is the smallest $\dR$ such that we can embed $K$ in $\R^\dR$?
% It is easy to prove that any such $K$ embeds into $\R^{2d+1}$.
% In general, this is as small as we can get: The Theorem of van Kampen and Flores tells us that there are $d$-dimensional simplicial complexes that do not embed into $\R^{2d}$.

It is well-known and not hard to show that a planar graph on $n$ vertices can have at most $3n-6$ edges.
We will be interested in higher-dimensional analogues of this result, i.e., in the maximal size of complexes of dimension $d\geq 2$ that are embeddable into a certain Euclidean space.

\pagebreak
More precisely, we study the following question:
\begin{que}
  For fixed $d$ and $\dR$, what is the maximal number of $d$-simplices for a complex on $n$ vertices that embeds into $\R^\dR$?
\end{que}
Because the complete $d$-complex embeds into $\R^{2d+1}$, for $\dR \geq 2d+1$ the maximal number of $d$-simplices one can get when embedding a complex on $n$ vertices into $\R^{\dR}$ is $\ueber{n}{d+1}= \Theta(n^{d+1})$. %For $\dR < 2d+1$ the complete complex gives a (trivial) upper bound.

For the case $d=2$ this yields $\Theta(n^3)$ for embeddability into $\R^5$. One can also show that a complex which embeds into $\R^3$ can have at most $O(n^2)$ triangles. What happens in $\R^4$ is an open question.
Apart from two related conjectures (\cite[Conjecture~27]{Kalai.2002}, \cite{Kalai.2008}) it seems this question has not been discussed in the prior literature.

We will address the general question for $d\leq r \leq 2d$. Our goal is to find lower and upper bounds. Chapter \ref{basics} will give a more detailed presentation of the problem.
Chapter \ref{lowerBounds} addresses the lower, Chapter \ref{upperBounds} the upper bounds.

We will need basic notions from quite a few areas of mathematics.
Instead of explaining all necessary basic concepts in one starting chapter, I decided to introduce things whenever they are needed as we go along .
%So Chapter \ref{basics} also contains introductions to the basic concepts, simplicial complexes and embeddings, needed to phrase the question.
%Chapter \ref{basics} will explore this problem a little bit further and also introduce the basic concepts, simplicial complexes and embeddings, needed to phrase the question. Other concepts, polytopes, hypergraphs, the probabilistic method, will be 

In Chapter \ref{lowerBounds} we achieve a lower bound of the order $\Omega(n^{\lceil\frac{\dR}{2}\rceil})$ for every instance of the problem by looking at examples of embeddable complexes. These are given by the boundary complexes of polytopes.
% the examples achieving the bound by the cyclic polytopes.
We will see that the cyclic polytopes yield the largest examples coming from polytopes, even from simplicial spheres.
We then show that in the case $\dR=2d$ these examples cannot be improved by simply adding further simplices.
% We will see that these are the largest examples coming from polytopes, even from simplicial spheres and also that these examples can not be improved by simply adding further simplices.
%Before all this, this chapter introduces polytopes and cyclic polytopes.

Chapter \ref{upperBounds} gives an upper bound for the case $\dR=2d$, where the lower bound is of the order $\Omega(n^d)$. The smallest interesting case here is $d=2$ and $\dR=4$.
We approach the question of how many $d$-simplices a complex embeddable into $\R^{2d}$ can have at most by looking for forbidden subcomplexes.
We exclude Schild's class of minimal non-embeddable complexes \cite{Schild.1993}, which seems to include all known examples of such complexes.

Turning to the methods of extremal hypergraph theory, we use a result of Erd\H{o}s \cite{Erdos.1964} on complete $k$-partite $k$-graphs with partition sets of fixed size.
This yields an upper bound of the order $O(n^{d+1-\frac{1}{3^d}})$, which improves the trivial upper bound of $\ueber{n}{d+1}=O(n^{d+1})$.
As Erd\H{o}s' result only estimates the extremal quantity in question, we then consider how much could be gained by a better estimate. We see that in this way the bound could not be improved to yield more than $O(n^{d+1-\frac{2(d+1)}{3^{d+1}-1}})$.
%by how much this bound could be improved

%A generalization of the theorem of van Kampen and Flores yields those.
%Then the problem can be tackled with the methods of extremal hypergraph theory.
%This gives an upper bound of $O(n^{d+1-\frac{1}{3^d}})$, which improves the trivial upper bound of ${n \choose {d+1}}$.
%Hypergraphs, the basics of extremal hypergraph theory and random hypergraphs are also presented.
%
\chapter{Embedding Simplicial Complexes}\label{basics}
This chapter introduces the main question that will be addressed in this thesis. We will first explore the notions needed to phrase the problem, namely simplicial complexes and their embeddings into Euclidean space. Then we will pose the question and discuss some first aspects.
\section{Simplicial Complexes}\label{simplicialComplexes}

Simplicial complexes are combinatorial objects that can be used to model certain topological spaces in a discrete setting.
The combinatorial concept of abstract simplicial complexes has a geometric counterpart, geometric simplicial complexes: subspaces of $\R^n$ that are built from simple building blocks. These give the connection to topology.
We will only be dealing with finite simplicial complexes.
Good sources for more information on this are, e.g., \cite{Matousek.2003} and \cite{Munkres.2005}.
Since we will rarely consider any other types of complexes (e.g., polytopal, $\Delta$- or CW-complexes) the term ``complex'' without further specification will always refer to a simplicial complex.

Let us begin with the definition of an abstract simplicial complex.
\begin{Def}[abstract simplicial complex]
 Let $V$ be a finite set.
 An \emph{abstract simplicial complex} on the vertex set $V = V(K)$ is a non-empty family $K \subseteq \mathcal{P}(V)$ of subsets of
 $V$ that is hereditary, i.e.:
 \[
  F \in K, G \subset F \Rightarrow G \in K.
 \]
 Members of $K$ are called \emph{simplices}.

 The \emph{dimension} of a simplex $F \in K$ is $\dim(F) \mathrel{\mathop:}= |F|-1$.
 The \emph{dimension} of $K$ is $\dim(K) \mathrel{\mathop:}= \max\left\{\dim(F)\;\middle|\; F \in K\right\}$.
 Simplices and complexes of dimension $d$ are called \emph{$d$-simplices} and \emph{$d$-complexes} respectively.
\end{Def}

Observe that, as a simplicial complex is non-empty, it will always contain the empty set.
To define the corresponding geometric concept we need some basic terminology:

\begin{Def}[affine independence, geometric simplex, face]\hspace{1em}\\
 Let $X=\{x_0,x_1,\ldots,x_n\} \subset \R^\dR$.
 \begin{enumerate}
  \item The set $X$ is called \emph{affinely independent} if for $\lambda_0,\lambda_1,\ldots,\lambda_n \in \R$ the equations $\sum_{i=0}^n\lambda_i=0$ and $\sum_{i=0}^n\lambda_ix_i=0$ imply that $\lambda_i=0$ for all $0 \leq i \leq n$.
  \item If $X$ is affinely independent, its convex hull $\sigma = \conv(X)$ is a \emph{(geometric) simplex} of dimension $n$ (or an $n$-simplex).
   The points $x_0,x_1,\ldots,x_n$ are the \emph{vertices} of $\sigma$.
   Note that an $n$-simplex has $n+1$ vertices.
  \item For an affinely independent set $X$, any subset $X'\subseteq X$ is also affinely independent and hence its convex hull is also a simplex. It is called a \emph{face} of $\sigma = \conv(X)$. In particular, every simplex has the empty set as a face.
 \end{enumerate}
\end{Def}
The affine hull of finitely many points $x_0,x_1,\ldots,x_n \in \R^\dR$
can be described as the set 
$\left\{\sum_{i=0}^n\lambda_ix_i\;\middle|\;\sum_{i=0}^n\lambda_i=1\right\}$.
Thus, for a set $X=\{x_0,x_1,\ldots,x_n\}$ to be affinely dependent means that at least one point in $X$, w.l.o.g. $x_n$, lies in the affine hull of the others. % points in $X$. 
Because this is equivalent to the linear dependence of the set $\{x_0-x_n,x_1-x_n,\ldots,x_{n-1}-x_n\}$,
this shows that the maximal size of an affinely independent set in $\R^\dR$ is $\dR+1$.
This is equivalent to asking for the affine hull of $X$ to have dimension $n$.

Let us consider some examples of simplices:
A $0$-simplex is just a point. Any two distinct points are affinely independent. A set of three points is affinely independent if the three points do not lie on a common line. Thus, any line segment can be considered as a $1$-simplex, a triangle as a $2$-simplex. A \mbox{$3$-simplex} is a tetrahedron, the convex hull of four points that do not lie in a common plane.
%In general, $n+1$ points in $\R^n$ are affinely independent if and only if they do not lie in a common hyperplane, and $n+1$ is the maximal size of an affinely independent set in $\R^n$.
%In general, $n$ points in $\R^\dR$ are affinely independent if and only if they do not lie in a common $(n-2)$-dimensional affine subspace. 
%
\begin{figure}[ht]
\centering
\includegraphics[scale=0.35]{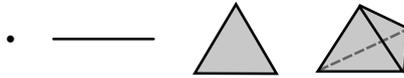}
\caption{Some simplices of low dimension.}\label{simplices}
\end{figure}

Simplices are the building blocks for geometric simplicial complexes:
\begin{Def}[geometric simplicial complex]
 A \emph{geometric simplicial complex} in $\R^\dR$ is a non-empty family $\Delta$ of geometric simplices in $\R^\dR$ that fulfills the following conditions:
 \begin{enumerate}
  \item 
   If $\sigma \in \Delta$ and $\sigma'$ is a face of $\sigma$, then $\sigma' \in \Delta$.
   \item
    If $\sigma_1, \sigma_2 \in \Delta$, then $\sigma_1 \cap \sigma_2$ is a face of both $\sigma_1$ and $\sigma_2$.
 \end{enumerate}
 The \emph{dimension} of a simplicial complex $\Delta$ is $\dim(\Delta) \mathrel{\mathop:}= \max\left\{\dim(\sigma)\;\middle|\;\sigma \in \Delta\right\}$; its vertex set
 $V(\Delta)$ consists of all vertices of simplices of $\Delta$.
\end{Def}

\begin{figure}[ht]
\centering
\includegraphics[scale=0.5]{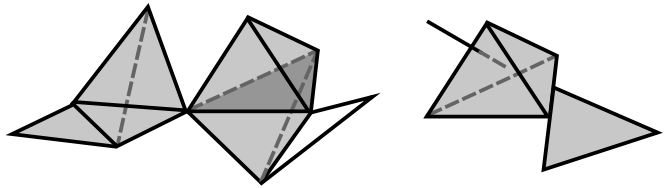}
\caption{A simplicial complex of dimension $3$ and a family of simplices that is not a simplicial complex.}
\end{figure}

A simple example for a geometric simplicial complex is the set of all faces of a simplex $\sigma$, called the \emph{boundary} $\partial\sigma$ of the simplex. Also the simplex itself can be considered as a simplicial complex: Take its boundary and add the simplex.
For a $d$-simplex, the corresponding abstract complex is simply the set of all subsets of its vertices, which is isomorphic to the power set of $[d]$ and is denoted by $\Delta_d$.

How are abstract and geometric simplicial complexes connected in general?
A geometric simplicial complex $\Delta$ gives rise to an abstract complex $K$ in a straight-forward way:
 The vertices of $K$ are just the vertices of $\Delta$: $V(K) = V(\Delta)$. A set $F \subset V(\Delta)$ forms a simplex of $K$ if $F$ is the vertex set of a simplex in $\Delta$. That means we just use the underlying set structure.
For any $K'$ that is isomorphic to $K$, the geometric complex $\Delta$ is called a \emph{geometric realization} of $K'$.

In this context ``isomorphic'' means the following:
%When are two abstract simplicial complexes isomorphic?
%
\begin{Def}[simplicial map, isomorphism]
 Let $K_1, K_2$ be two abstract simplicial complexes.
 A map \mbox{$f\!:\!V(K_1) \to V(K_2)$} between the vertex sets of the two complexes is a \emph{simplicial map} from $K_1$ to $K_2$ if $f(F) \in K_2$ for all $F \in K_1$. 
A simplicial map from $K_1$ to $K_2$ is an \emph{isomorphism} if it is bijective on the vertex sets and if its inverse is also simplicial, i.e., a map \mbox{$f\!:\!V(K_1) \to V(K_2)$} that induces a bijection between the simplices of $K_1$ and $K_2$.
% An \emph{isomorphism} between $K_1$ and $K_2$ is a simplicial map from $K_1$ to $K_2$ that is bijective on the vertex sets and such that its inverse is also simplicial, i.e.\ such a map induces a bijection between the simplices of $K_1$ and $K_2$.
%  Let $K_1, K_2$ be two abstract simplicial complexes.
%  A \emph{simplicial map} from $K_1$ to $K_2$ is a map $f\!:\!V(K_1) \to V(K_2)$ such that $f(F) \in K_2$
%    for all $F \in K_1$. 
%  An \emph{isomorphism} between $K_1$ and $K_2$ is a bijective simplicial map from $K_1$ to $K_2$ such that its inverse is also simplicial.
\end{Def}

For any abstract simplicial complex $K$ with $|V(K)|=n$ we can find a geometric realization in $\R^{n-1}$ by identifying its $n$ vertices with the vertices of a geometric $(n-1)$-simplex in $\R^{n-1}$.
We can consider two geometric simplicial complexes as (combinatorially) equivalent if they are geometric realizations of isomorphic abstract complexes. With this notion, all geometric realizations of an abstract complex are equivalent, even though they might lie in surrounding spaces of different dimensions.

The union of all simplices in a geometric simplicial complex is a topological space. We will see that equivalent complexes yield homeomorphic spaces.
\begin{Def}[polyhedron of a geometric simplicial complex, triangulation]
 Let $\Delta$ be a geometric simplicial complex in $\R^\dR$.
 \begin{enumerate}
  \item
   The \emph{polyhedron} of $\Delta$ is the topological space $||\Delta||$ on the set
   $
    \bigcup_{\sigma \in \Delta} \sigma
   $
   endowed with the following topology: A set $O \subseteq ||\Delta||$ is open if and only if $O \cap \sigma$ is open in $\sigma$ for every $\sigma \in \Delta$, where every $\sigma$ carries the subspace topology inherited from $\R^\dR$.
  \item
   If $X$ is a topological space that is homeomorphic to $||\Delta||$, we call $\Delta$ a \emph{triangulation} of $X$.
 \end{enumerate}
\end{Def}

Note that in the case of finite complexes $\Delta$, which we are considering, the topology on $||\Delta||$ agrees with the subspace topology inherited from $\R^\dR$.

%Note that for the general case of infinite complexes, one has to specify a different kind of topology which for the finite case agrees with the subspace topology.
% but in the finite case the subspace topology agrees with this one.

Simple examples for spaces that admit a triangulation are the $d$-ball, which is homeomorphic to the polyhedron of any $d$-simplex, and the $d$-sphere, homeomorphic to the polyhedron of the boundary of a $d$-simplex.

With an abstract simplicial complex $K$ we can now associate $||\Delta||$ for a geometric realization $\Delta$ of $K$.
A complex might (and will) have several realizations, do these yield different topological spaces?
We will see that this cannot happen: Any two isomorphic abstract complexes give rise to homeomorphic spaces.

From a simplicial map between two abstract simplicial complexes we get a map between the polyhedra of the corresponding geometric complexes:
\begin{Def}[affine extension of a simplicial map]
 Let $\Delta_1$ and $\Delta_2$ be two geometric simplicial complexes and let $K_i$ be the abstract complex corresponding to
 $\Delta_i$.

 For any simplicial map $f\!:\!V(K_1) \to V(K_2)$ from $K_1$ to $K_2$ its \emph{affine extension}
\[
 ||f||\!:\!||\Delta_1|| \to ||\Delta_2||
\]
is defined by
\[
 ||f||(x) = \sum_{i=0}^{k}\lambda_i f(v_i)
\]
 for $x=\sum_{i=0}^{k}\lambda_i v_i \in \Delta_1$.
 Here we use that every point in $||\Delta_1||$ has a unique representation as a convex combination of the vertices of the minimal simplex in $\Delta_1$ that contains it.
\end{Def}
This map is continuous and for isomorphic complexes it is a homeomorphism:
\begin{prop}[{e.g., \cite[Proposition 1.5.4]{Matousek.2003}}]
 For any simplicial map $f$ its affine extension $||f||$ is continuous.
 If $f$ is an isomorphism, $||f||$ is a homeomorphism.
\end{prop}
Thus, if $\Delta_1$ and $\Delta_2$ are two geometric realizations of the same abstract simplicial complex $K$, the map $||\id_K||$ certifies that $||\Delta_1||$ and $||\Delta_2||$ are homeomorphic. This shows that, up to homeomorphism, an abstract simplicial complex $K$ gives rise to a unique topological space, the \emph{polyhedron} of $K$, which we will denote by $||K||$.

We can therefore stop to distinguish between abstract and geometric complexes and will from now on only talk of abstract complexes and their realizations.

We now introduce some more notions connected to simplicial complexes.
In most of our considerations we will restrict our attention to complexes consisting of the faces of simplices of a fixed dimension.
These are called ``pure'':
\begin{Def}[pure complex]
 A $d$-dimensional simplicial complex $K$ is \emph{pure} if every simplex of $K$ is a face of some $d$-simplex in $K$. Consequently, all maximal simplices of a pure complex $K$ have the same dimension.
\end{Def}

\begin{figure}[ht]
\centering
\includegraphics[scale=0.5]{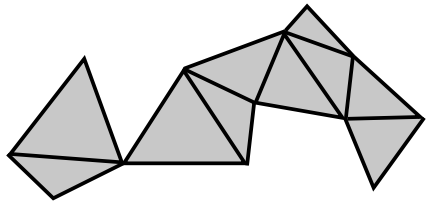}
\caption{A pure $2$-complex.}
\end{figure}

In what is to come we will often consider subcomplexes of simplicial complexes. One of our most important examples will be the $k$-skeleton, a special subcomplex, of a certain complex.
\begin{Def}[subcomplex, $k$-skeleton]
 Let $K$ be a simplicial complex.
\begin{enumerate}
  \item 
   A \emph{subcomplex} of $K$ is a subset $L$ of $K$ such that $L$ is also a simplicial complex,
   i.e., $(F \in L, G \subset F \Rightarrow G \in L).$
  \item
   For $k \leq \dim(K)$ the \emph{$k$-skeleton} of $K$ is the subcomplex consisting of all simplices of dimension at most $k$:
   \[
    K^{\leq k} \mathrel{\mathop:}= \left\{F \in K \;\middle|\; \dim(F) \leq k \right\}.
   \]
 \end{enumerate}
\end{Def}

In addition to subcomplexes, we will also consider minors of simplicial complexes. For the definition of minors and also for several notions from piecewise-linear topology, we will need the concept of a subdivision of a simplicial complex.

\begin{Def}[subdivision]
A simplicial complex $K'$ is a \emph{subdivision} of $K$ if there are geometric realizations $\Delta$ and $\Delta'$ of $K$ and $K'$ such that $||\Delta||=||\Delta'||$ and each simplex of $\Delta'$ is contained in some simplex of $\Delta$.
\end{Def}

As the last definition in this section, we will now introduce an operation for topological spaces, the join, which has certain advantages over the Cartesian product when working with simplicial complexes (e.g., \cite[Section 4.2]{Matousek.2003}).
%In addition to the Cartesian product there exists another type of operation for topological spaces, the join, which has certain advantages when working with simplicial complexes (e.g.\ \cite[Section 4.2]{Matousek.2003}).
The Cartesian product of two (geometric) simplices of dimension at least $1$ is, e.g., not a simplex, whereas the join of two simplices is. 
In the next section, we will encounter a family of complexes that consists of joins of certain complexes. As we will simply state results on their topological properties and only study combinatorial aspects, we now just introduce the combinatorial version of this operation: the join of two abstract simplicial complexes.
\begin{Def}[join]
 Let $K_1, K_2$ be two abstract simplicial complexes.
 %The \emph{join} $K_1*K_2$ of $K_1$ and $K_2$ is the simplicial complex defined by the following data:
 The simplicial complex $K_1*K_2$ defined by the following data is called the \emph{join} of $K_1$ and $K_2$:
\begin{itemize}
 \item As vertex set it has the disjoint union of the vertex sets of $K_1$ and $K_2$:
 \[
  V(K_1*K_2) \mathrel{\mathop:}= V(K_1) \mathbin{\dot\cup} V(K_2).
 \]
 \item Its simplices are given by disjoint unions of simplices in $K_1$ and simplices in $K_2$:
 \[
  K_1*K_2 \mathrel{\mathop:}= \left\{F_1 \mathbin{\dot\cup} F_2 \;\middle|\; F_1 \in K_1, F_2 \in K_2\right\}.
 \]
\end{itemize}
The disjoint union $S_1 \mathbin{\dot\cup} S_2$ of two sets $S_1$ and $S_2$ can, e.g., be considered as the set 
$(S_1\times\{1\} \cup S_2\times\{2\}).$
Note that for a $d_1$-complex $K_1$ and a $d_2$-complex $K_2$ the dimension of the join $K_1*K_2$ is \[\dim(K_1*K_2)=d_1+d_2+1.\]
\end{Def}

\section{Embeddings}\label{embeddings}
%Embedding simplicial complexes into Euclidean space
%
%Quellen: Matousek,
%
In this section, we study embeddings of simplicial complexes into Euclidean space.
We will discuss different types of embeddability and consider the minimal dimension in which a fixed complex can be embedded.

As a start remember the definition of an embedding for general topological spaces:

\begin{Def}[embedding]
 Let $X$ and $Y$ be topological spaces.

 An \emph{embedding} of $X$ into $Y$ is a map $f\!:\!X \to Y$ that is a homeomorphism onto its image $f(X)$.
 This means that $f\!:\!X \to f(X)$ with the induced topology on $f(X)\subseteq Y$ is bijective, continuous and the inverse function $f^{-1}$ is also continuous.
\end{Def}

In this thesis we consider embeddings of the form $f\!:\!||K|| \to \R^{\dR}$ for $d$\mbox{-}dimen\-sional simplicial complexes $K$.
From now on, $d$ will always refer to the dimension of the complex, whereas $\dR$ will be the dimension of the Euclidean space in which we want to embed the complex.
If $||K||$ embeds into $\R^{\dR}$ we denote this by $||K|| \hookrightarrow \R^{\dR}$ and often just say that $K$ embeds into $\R^{\dR}$.

Geometric realizations are of course examples of such embeddings. They can be considered as linear embeddings, where \emph{linear} for a map of a simplicial complex into $\R^{\dR}$ means that it is linear on each simplex.

A \emph{piecewise linear (PL)} embedding of a simplicial complex $K$ into $\R^{\dR}$ is a map $f\!:\!||K|| \to \R^{\dR}$ that is a linear map of some subdivision $K'$ of $K$ into $\R^{\dR}$ as well as an embedding.

\begin{figure}[ht]
\centering
\includegraphics[scale=0.4]{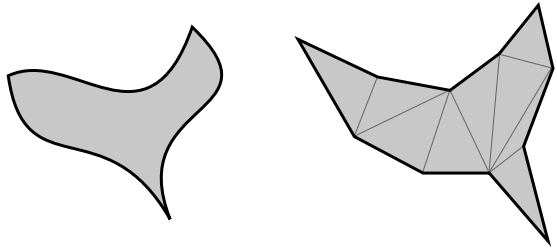}
\caption{A topologically and a piecewise linearly embedded $2$-simplex.}
\end{figure}

A complex that admits a topological embedding into some $\R^{\dR}$ does not necessarily also admit a PL (or a linear) embedding into this space.
\label{PLtop}
In \cite{Matousek.2008} Matou\v{s}ek, Tancer and Wagner study the computational complexity of deciding whether a given $d$-dimensional simplicial complex embeds piecewise linearly into $\R^{\dR}$. They also compare linear, PL and topological embeddability; we repeat parts of their discussion here.

There are some cases where there is no difference:
Every $d$-dimensional simplicial complex that embeds topologically into $\R^{\dR}$ for some $\dR$ with $\dR-d \geq 3$ is also PL embeddable into $\R^{\dR}$ \cite{Bryant.1972}.
The same is true for $d=2$ and $\dR=3$ \cite[p.858]{Matousek.2008} and also for planar graphs (i.e., $1$-dimensional simplicial complexes that embed into $\R^2$) (e.g., \cite[p.21]{Bollobas.2002}).
A well-known fact from graph theory, F\'ary's Theorem (e.g., \cite[p.246/247]{West.2001}), states that for every simple planar graph there even exists a straight-line (i.e., linear) embedding. 
%Original-Quelle fuer Fary's Thm??

But there are examples which show that in general these embeddability properties are not the same:
There is a $4$-dimensional complex that embeds topologically, but not PL into $\R^5$ \cite[p.858]{Matousek.2008}.
In \cite{Brehm.1983} Brehm presents a $2$-dimensional simplicial complex $K$ (a triangulated M\"obius strip) that embeds into $R^3$ but not linearly.
It has also been shown that, while trivially embedding topologically into $\R^3$, no triangulation of a surface of genus $6$ using only $12$ vertices admits a linear embedding in $\R^3$: \cite{Bokowski.2000} proves this for one example, \cite{Schewe.2006} for all.
For higher dimensions, Brehm and Sarkaria \cite{Brehm.1992} showed that for every $d\geq2$ and every $\dR$ such that $d+1 \leq \dR \leq 2d$ there is a finite $d$\mbox{-}dimensional complex $K$ that admits a topological but no linear embedding into $\R^{\dR}$. For given $l$, it is even possible to construct $K$ such that $K$ can be embedded into $\R^{\dR}$, but the $l$-th barycentric subdivision $\sd^l(K)$ cannot be embedded linearly into $\R^{\dR}$.

Although our main interest will lie in topological embeddings, we will occasionally point out where our results are relevant to questions concerning the other embeddability properties.

Let us now consider the minimal dimension in which we can embed a given complex.
We have already seen that a simplicial complex on $n$ vertices always has a geometric realization in $\R^{n-1}$.
Here is a better result which depends only on the dimension of the complex, not on the number of vertices:
\begin{thm}[{e.g., \cite[Theorem 1.6.1]{Matousek.2003}}]\label{geomRealization}
 Every finite $d$-dimensional simplicial complex $K$ has a geometric realization $\Delta$ in $\R^{2d+1}$.
\end{thm}
To prove this we will use the moment curve which will also be of use later on.
\begin{Def}[moment curve]
 The curve $\gamma\!:\!\R \rightarrow \R^{\dR}$ defined by 
\[\gamma(t) = (t,t^2,\ldots,t^{\dR})\]
is called the \emph{moment curve} in $\R^{\dR}$.
\end{Def}
\begin{lem}[{e.g., \cite[Lemma 1.6.2]{Matousek.2003}}] \label{momentcurve}
 Any $\dR+1$ distinct points on the moment curve in $\R^{\dR}$ are affinely independent.
\end{lem}
\begin{proof}
 To check affine dependence of points $x_0,x_1,\ldots,x_{\dR}$ we have to solve the system of linear equations
 $\sum_{i=0}^{\dR} {\lambda}_i x_i = 0$, $\sum_{i=0}^{\dR} {\lambda}_i = 0$.
 For points $x_i$ on the moment curve this leads to determining the rank of the matrix 
 \[
 \begin{pmatrix}
  1& 1& \ldots& 1\\
  t_0& t_1& \ldots& t_{\dR}\\
  \vdots& \vdots& & \vdots\\
  t_0^{\dR}& t_1^{\dR}& \ldots& t_{\dR}^{\dR}
 \end{pmatrix},
 \]
 a Vandermonde matrix.
 Its determinant is known to be $\prod_{0 \leq i < j \leq \dR}\left(t_j - t_i\right)$ and therefore non-zero for pairwise distinct $t_i$.
 Thus, the zero vector is the only solution.
\end{proof}
\begin{proof}[Proof of Theorem \ref{geomRealization}]
 Let $K$ be a finite $d$-dimensional simplicial complex.

 Place its vertices on the moment curve in $\R^{2d+1}$ via some $f\!:\!V(K) \rightarrow \R^{2d+1}$.
 Because a simplex in $K$ has at most $d+1 \leq 2d+2$ vertices, it corresponds to an affinely independent set.
 Thus, by taking the convex hulls of sets corresponding to simplices in $K$, we get a collection of geometric simplices.

 To see that this collection is a simplicial complex, we have to show that for two simplices $F_1$ and $F_2$ in $K$ the intersection
 $\sigma_1 \cap \sigma_2$ of the corresponding simplices $\sigma_i = \conv(f(F_i))$ is a face of both $\sigma_1$ and $\sigma_2$.

 But $|F_1 \cup F_2| \leq 2d+2$, so the set of involved vertices is affinely independent.
 This means $\sigma_1$ and $\sigma_2$ are faces of a bigger simplex and their intersection is a face of both.
\end{proof}
Thus, every $d$-dimensional complex can be embedded in $\R^{2d+1}$.
The following theorem by van Kampen \cite{vanKampen.1932a, vanKampen.1932b} and Flores \cite{Flores.1933, Flores.1934} shows that for some complexes this is the best possible dimension. A modern treatment can be found in \cite{Matousek.2003}.
\begin{thm}[Van Kampen-Flores Theorem]\label{vK-Fl}
Let $d \geq 1$.

Then $K\mathrel{\mathop:}=(\Delta_{2d+2})^{\leq d}$ does not embed into $\R^{2d}$. %cannot be embedded?????

More precisely, for every continuous map $f\!:\!||K|| \rightarrow \R^{2d}$ there exist two disjoint simplices $F_1, F_2 \in K$
such that $f(F_1) \cap f(F_2) \neq \emptyset$.
\end{thm}
Of course there are $d$-dimensional complexes that can be embedded in some $\R^{\dR}$ with $\dR \leq 2d$.
A trivial example is the $d$-simplex which embeds into $\R^d$.
Later we will consider (the $d$-dimensional) boundaries of simplicial $(d+1)$-polytopes which lie in $\R^{d+1}$.

We will be especially interested in $d$-dimensional simplicial complexes that admit an embedding into $\R^{2d}$.
For $d=1$ these are planar simple graphs for which there are well-known characterizations
via forbidden subgraphs and minors (e.g., \cite[Section 6.2]{West.2001}):
\begin{thm}[Kuratowski's Theorem]\label{Kuratowski}
A simple graph is planar if and only if it does not contain a subdivision of $K_5$ or of $K_{3,3}$ as a subgraph.
\end{thm}
\begin{thm}[Wagner's Theorem]\label{Wagner}
A simple graph is planar if and only if it does not have $K_5$ or $K_{3,3}$ as a minor.
\end{thm}
Here a \emph{minor} of a graph $G$ is a graph $H$ that can be obtained from $G$ by deleting and/or contracting edges of $G$. 
A \emph{subdivision} of $G$ is a graph $G'$ that is obtained by replacing edges of $G$ with pairwise internally-disjoint paths. With $K_5$ we denote the complete graph on $5$ vertices, and $K_{3,3}$ refers to the complete bipartite graph on two vertex sets of three elements each. (See Figure \ref{kfive-kthreethree})

\begin{figure}[htbp]
\centering
\includegraphics[scale=0.7]{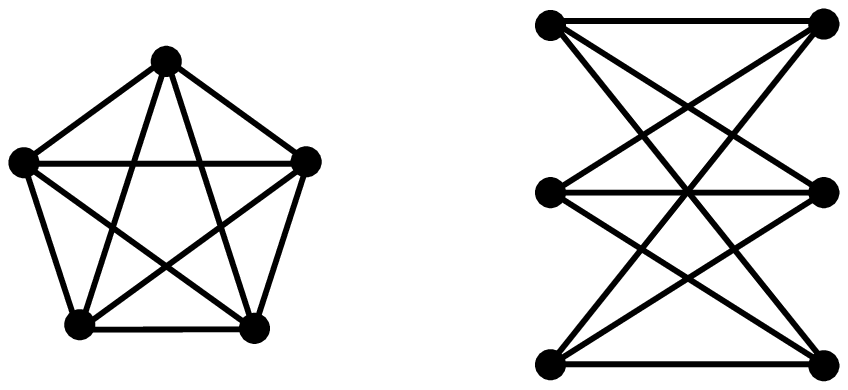}
\caption{The graphs $K_5$ and $K_{3,3}$.}\label{kfive-kthreethree}
\end{figure}

The graphs $K_5$ and $K_{3,3}$ are minimal non-planar graphs: All of their subgraphs are planar. Moreover, we just saw that these two graphs and all of their subdivisions are the only minimal non-planar graphs.
Any other graph that does not embed into the plane contains a (proper) subgraph that is homeomorphic to one of the two and is hence non-planar.

In the case $\dR=2d$ for $d \geq 2$ an example of a minimal non-embeddable complex is $(\Delta_{2d+2})^{\leq d}$, the complex appearing in the Theorem of van Kampen and Flores.
All of its subcomplexes embed into $\R^{2d}$ because $(\Delta_{2d+2})^{\leq d} \setminus F$, where $F$ is a maximal simplex, is the complex on $2d+3$ vertices consisting of all possible $d$-simplices except for one.
This is the $d$-skeleton of the boundary complex of $C_{2d+1}(2d+3)$, the cyclic $(2d+1)$-polytope on $2d+3$ vertices, as we will, e.g., see in the proof of Lemma \ref{zweiDeeMinor}.
%(See Section \ref{Polytopes} for an introduction into polytopes, Section \ref{CyclicPolytopes} for cyclic polytopes.)
% All of its subcomplexes embed into $\R^{2d}$ because $(\Delta_{2d+2})^{\leq d} \setminus F$, where $F$ is a maximal simplex,
% %the complex consisting of all possible $d$-simplices except for one on $2d+3$ vertices
% is the $d$-skeleton of the boundary complex of $C_{2d+1}(2d+3)$, the cyclic $(2d+1)$-polytope on $2d+3$ vertices.
% (See Section \ref{Polytopes} for an introduction into polytopes, Section \ref{CyclicPolytopes} for cyclic polytopes.)

Further classes of minimal non-embeddable complexes in higher dimensions were presented by Gr\"unbaum \cite{Gruenbaum.1969}, Zaks \cite{Zaks.1969b}, Sarkaria \cite{Sarkaria.1991} and Schild \cite{Schild.1993}.
Schild's class of complexes, presented in the theorem below, contains all of these examples.
\begin{Def}[nice complex]
 Let $K$ be a simplicial complex on at most $n$ vertices, with a vertex set identified with some $V\subseteq[n]$.
 We call $K$ \emph{nice} on $[n]$ if the following condition holds:
 \[
  F \subset [n] \Rightarrow F \in K \text{ or } [n] \setminus F \in K, \text{ but not both.}
 \]
\end{Def}
The complex from the Theorem of van Kampen and Flores (Theorem \ref{vK-Fl}) is one example of a nice complex: For $F\subset [2d+3]$ with $|F|\geq d+2$, the complement $[2d+3] \setminus F$ has $\leq d+1$ elements and is thus in $(\Delta_{2d+2})^{\leq d}$.
In Schild's paper the definition of a nice complex is slightly different from the definition we are using.
Where we use $[n]$, he instead puts $V(K)$, the vertex set of $K$ which is a subset of $[n]$ in our setting. The condition on a nice complex becomes
 \[
  F \subset V(K) \Rightarrow F \in K \text{ or } V(K) \setminus F \in K, \text{ but not both.}
 \]
The only case of a complex that is nice on $[n]$ in which $V(K) \neq [n]$ is the $(n-2)$-simplex which Schild in his paper adds as an exception by looking at the simplex with an additional ``virtual'' vertex.

Here's a short proof for this being the only case:
If there exists $x \in [n]$ such that $\{x\} \notin K$, then $[n]\setminus \{x\} \in K$ because $K$ is nice on $[n]$.
This means $[n]\setminus \{x\}$ is the only maximal simplex of $K$.

With this definition we can now present Schild's class of non-embeddable complexes:
\begin{thm} [{Schild~\cite[Theorem 3.1]{Schild.1993}}] \label{joins of nice complexes}\xdef\savedtheoremnumber{\thethm}
 Let $K_1, K_2,\ldots,K_s$ be simplicial complexes and suppose there are $n_1, n_2,\ldots,n_s \in \mathbb{N}$ such that $K_i$ is nice on $[n_i]$.
 Then $K = K_1*K_2*\ldots*K_s$ is not embeddable in $\R^{\dR}$ for ${\dR}=\left( \sum_{i=1}^s n_i \right)-s-2$.

 Except for the following two cases the complex $K$ is minimal non-embeddable,
 i.e., every proper subcomplex is linearly embeddable in $\R^{\dR}$:
 \begin{enumerate}
  \item All $K_i$ are simplices.
  \item $K=K_1$ is the boundary of a simplex with an additional vertex ($s=1$). 
 \end{enumerate}
\end{thm}
% \begin{thm} [{Schild~\cite[Theorem 3.1]{Schild.1993}}] \label{joins of nice complexes}
%  Let $K_1, K_2,\ldots,K_r$ be simplicial complexes and $n_1, n_2,\ldots,n_r \in \mathbb{N}$ such that $K_i$ is nice on $[n_i]$.
% 
%  Then $K = K_1*K_2*\ldots*K_r$ is not embeddable in $\R^n$ for $n=n_1+n_2+\ldots+n_r-r-2$.
% 
%  Except for the following two cases the complex $K$ is minimal non-embeddable,
%  i.e.\ every proper subcomplex is linearly embeddable in $\R^n$:
%  \begin{enumerate}
%   \item All $K_i$ are simplices.
%   \item $K=K_1$ is the boundary of a simplex with an additional vertex ($r=1$). 
%  \end{enumerate}
% \end{thm}
%
Note that this implies the statement of the Theorem of van Kampen and Flores (Theorem \ref{vK-Fl}) that the complex $(\Delta_{2d+2})^{\leq d}$ does not embed into $\R^{2d}$, as $(\Delta_{2d+2})^{\leq d}$ is nice on $[2d+3]$.

Can these complexes play the same role as $K_5$ and $K_{3,3}$ for planar graphs, i.e., do they describe all complexes that are embeddable in the respective dimension?

Besides the case $d=1$, $\dR=2$ of planar graphs there are other classes of complexes that can be characterized by a finite set of forbidden subcomplexes or minors.
Halin and Jung in \cite{Halin.1964} give a characterization for $2$-dimensional complexes that embed in $\R^2$ by $7$ forbidden subcomplexes.
A sufficient condition for embeddability of $2$-complexes in $\R^3$ is given in \cite[Corollary 5.1]{Matousek.2008}.

For the case $\dR = 2d$ and $d \geq 2$ no such characterization is possible.
%robertson-seymour-type?
Zaks (\cite{Zaks.1969a} for $d>2$) and Ummel (\cite{Ummel.1973} for $d=2$) showed that there are infinitely many pairwise
non-homeomorphic $d$-complexes each of which does not embed into $\R^{2d}$
while all of their proper subcomplexes are (even linearly) embeddable.

There is also a concept of minors for these higher dimensional cases, inspired by graph minors.
In \cite{Nevo.2007} Nevo introduces minors for finite simplicial complexes, establishes a connection with embeddability
and also shows that this concept cannot yield a characterization of embeddable graphs: \label{minors}
He proves that for any $d \geq 2$ there exist infinitely many pairwise non-homeomorphic simplicial complexes of dimension $d$
that do not embed in the $2d$-sphere while all of their proper minors are embeddable.
We will introduce his concept of a minor in Section \ref{defMinor}.

This concludes our general explorations of the embeddability of simplicial complexes. We will now turn to the question of the size of embeddable complexes.
%thm: non-embeddability of minor & vK-Obstr \neq 0 => non-embeddability of complex ?? 

\section{Size of Embeddable Complexes}%maximum number of faces?
\label{size}
This section introduces the main topic of this thesis: the maximal size of complexes that are embeddable into a certain Euclidean space.
More precisely, we will address the following question:
% We want to know how big complexes which embed into some Euclidean space of a certain dimension can get.
% More precisely, we are interested in the following question:

%\vspace{0.5em}
%{\it
% \noindent
\begin{que}
Let $K$ be a simplicial complex of dimension $d$ on $n \geq r+1$ vertices that admits a (topological) embedding into $\R^{\dR}$ for $d \leq \dR \leq 2d$.
How many $d$-simplices (in terms of $n$, $d$ and $r$) can $K$ contain at most?
\end{que}

To be able to phrase the question shorter and more formally, we introduce the following notation for the numbers of simplices in a complex:
\begin{Def}[$f$-vector of a simplicial complex]
For a $d$-complex $K$ and $0\leq i\leq d$ we denote the number of $i$-simplices of $K$ by $f_i(K)$. The \emph{$f$-vector} of $K$ is the vector 
\[
f(K)=(f_0(K),f_1(K),\ldots,f_d(K)). 
\]
\end{Def}
With this we can now rephrase our question:
\begin{prob}\label{dieFrage}
For fixed $d,\dR \geq 1$ such that $d \leq \dR \leq 2d$ and $n \geq r+1$, what is (in terms of $n$, $d$ and $r$)
\[
 \max\left\{f_d(K) \;\middle|\; \dim(K)=d,\, |V(K)|=n,\, ||K|| \hookrightarrow \R^{\dR} \right\}?
\]
\end{prob}

In his blog \cite{Kalai.2008}, Gil Kalai proposes a conjecture, which he attributes to Sarkaria, concerning this question: 
\begin{conj}[{Kalai, Sarkaria \cite{Kalai.2008}}]\label{Conj1}
 Let $K$ be a $2$-dimensional simplicial complex with $f_2(K) \geq 4 f_1(K)$. Then $K$ cannot be embedded to $\R^4$.
\end{conj}
The type of embeddability, linear, piecewise-linear or topological, is not specified in this conjecture.
Note that, while we will be studying bounds depending on the number of vertices of a complex, here the number of maximal simplices is bounded in terms of the number of edges.
Kalai also remarks that there are similar conjectures for higher dimensions. At the end of this section, we will discuss Conjecture \ref{Conj1} in a little more detail.

In Section \ref{conjecture} we will look into a further conjecture by Kalai and Sarkaria on the shifted complexes of embeddable complexes \cite[Conjecture~27]{Kalai.2002} that yields a similar conjecture on this question.
%Apart from these references, it seems that Problem \ref{dieFrage} has not been discussed in the prior literature.
Problem \ref{dieFrage} is also discussed briefly in \cite[Notes in Section 5.1]{Matousek.2003}. A treatment of the case $d=2$ and $r=3$, which we will study later in this section, can be found in \cite{Dey.1994}.
The case $r=d$ for linear embeddings is solved in \cite{Dey.1998}, where an upper bound for the case $r=d+1$ for linearly embeddable complexes is also presented.
\vspace{1em}

We will now first explore some details of the question. Then we answer it for low values of $d$ and $\dR$,  where it is directly tractable, and present Dey's and Pach's results on the cases $r=d$ and $r=d+1$ for linear complexes \cite{Dey.1998}, before we turn to a discussion of the already mentioned conjectures on Problem~\ref{dieFrage}.

As a first remark, observe that, because $f_i(K)\leq \ueber{n}{i+1}$ for all simplicial complexes $K$ on $n$ vertices and all $i$, the set 
\[
\left\{f_d(K) \;\middle|\; \dim(K)=d,\, |V(K)|=n,\, ||K|| \hookrightarrow \R^{\dR} \right\}
\]
is finite, and we can thus consider its maximum.

Furthermore, note that we can restrict our attention to pure $d$\mbox{-}dimen\-sional complexes:
Embeddability of a complex $K$ implies embeddability for all subcomplexes $K'$ of $K$. Thus, for any embeddable complex $K$ with the maximum number of $d$-simplices, the pure complex 
%$\{F \in K| F \subseteq F' \text{ for some } F'\in K \text{ with } \dim(F') = d \}$
\[
 \left\{F \subseteq F'\;\middle|\; F'\in K, \dim(F') = d \right\} 
\]
also yields an example that has the same number of $d$-simplices.

If a simplicial complex embeds into $\R^{\dR}$, it can also be embedded into $S^{\dR}$ using inverse stereographic projection.
For $d<\dR$ the converse is also true:
An embedding of a $d$-dimensional complex $K$ into some $S^{\dR}$ with $d<\dR$ cannot be surjective
(otherwise it would be a homeomorphism between $K$ and $S^{\dR}$).
So via the stereographic projection a complex that admits such an embedding can also be embedded into $\R^{\dR}$.
This means we can (and will sometimes) consider embeddability into $S^{\dR}$ instead of $\R^{\dR}$ if $d \neq \dR$.

Why do we restrict the question to $\dR$ with $d \leq \dR \leq 2d$?
We saw that any $d$\mbox{-}complex is embeddable into $\R^{2d+1}$.
The complete $d$-complex on $n$ vertices (with all possible simplices) thus shows that for any dimension $\dR \geq 2d+1$ the answer to our question is $\ueber{n}{d+1}$, the maximal possible one.
For $n \geq 2d+3$ the complete complex does not embed into $\R^{2d}$ by the theorem of van Kampen and Flores. Because a $d$-complex cannot be embeddded into $\R^{d-1}$, the interesting range for $\dR$ hence lies between $d$ and $2d$.

It is moreover sufficient to ask for $n$ to be at least $r+1$, as the complete $d$-complex on $n$ vertices is the $d$-skeleton of the $(n-1)$-simplex, which clearly embeds into $\R^{\dR}$ for $n \leq r+1$. 

As mentioned in Section \ref{PLtop}, in some cases the existence of a topological embedding implies the existence of a PL embedding: For $d > 2$ this range is
$d+3 \leq \dR \leq 2d$, for $d=2$ it is the case $\dR = 3$.
In all other cases, asking for the complex to embed piecewise linearly (or linearly) and not only topologically might change the outcome of the question.
\vspace{1em}

This concludes our discussion of details concerning Problem \ref{dieFrage}.
After considering directly tractable cases for $d=1$, $2$, $3$ in this section, we will set out to find lower and upper bounds in higher dimensions, where the question seems to become more complicated.
We present lower bounds for all instances of the problem in Chapter \ref{lowerBounds} and upper bounds for the cases $r=2d$ in Chapter \ref{upperBounds}.
As linear or PL embeddability implies topological embeddability, the upper bounds naturally also apply for the variations of the question mentioned above.
The lower bounds do likewise because they come from linearly embeddable examples.

%Both bounds also apply for the variations of the question mentioned above: The lower bound comes from linearly embeddable examples and 
% The lower bound we present in Chapter \ref{lowerBounds} for all instances of the problem comes from linearly embeddable examples and would thus also apply for the variations of the question mentioned above.
% For the cases $r=2d$ we find an upper bound in Chapter \ref{upperBounds}. As linear or PL embeddability implies topological embeddability, this upper bound 

% as far as we know blabla does not appear in the literature so far
%...in the prior literature
%For higher dimensions the question seems to get less tractable. The following chapters will be devoted to finding lower and upper bounds in various cases.

Let us now study the directly tractable cases. As examples we will use the cyclic polytopes, even though we only introduce polytopes in general and cyclic polytopes in particular in Sections \ref{Polytopes} and \ref{CyclicPolytopes}. The face numbers of the cyclic polytopes we are using here will be calculated in Section \ref{simplicialpolytopes}.

\subsection*{Graphs}
Simplicial complexes of dimension $d = 1$ are simple graphs.
As all graphs embed into $\R^3$, we will consider $\dR$ as either $1$ or $2$. In both cases there is no difference between topological, PL and linear embeddability.

\begin{lem}
 A simple graph on $n$ vertices that embeds into $\R$ has at most $n-1$ edges.
\end{lem}
\begin{proof}
If a graph is embeddable into $\R$, it is a disjoint union of paths.
\end{proof}

The path of length $n-1$ proves this bound for $r=1$ to be tight.
For $r=2$ we consider planar graphs, for which the maximum number of edges is known:

\begin{lem}[{e.g., \cite[Theorem 6.1.23]{West.2001}}]
\label{sizePlanarGraphs}
 A simple planar graph on $n \geq 3$ vertices has at most $3n-6$ edges.
\end{lem}
\begin{proof}
 Let $G$ be a simple planar graph with $n \geq 3$ vertices and $e$ edges.

 Euler's formula states that $n-e+f=2$ for any crossing-free drawing of $G$, where $f$ is the number of faces.
 (A face is a maximal region of the plane which contains no point used in the drawing.)

 If $G$ is connected and not a path of length 2, every face has at least $3$ edges on its boundary.
%is bounded by at least 3 edges.
 Since every edge lies in the boundary of exactly two faces we get 
\[
 2 e = \sum_{i=1}^{f}F_i \geq 3f
\]
 where $F_i$ is the number of edges in the boundary of the $i$-th face.
 With Euler's formula this yields 
\[
 3n-3e+2e = 3n-e \geq 6.
\]

 If $G$ is not connected, one can simply add edges to get a connected graph with more edges but the same number of vertices.
 For the path of length 2 the statement is obviously true.
\end{proof}
Graphs attaining this bound are maximal planar in the sense that adding an edge will make them non-planar.
All of their faces (even the outer one) are triangles.

\subsection*{2-Complexes}
For $d=2$ we consider simplicial complexes that are collections of triangles and look at $r$ in the range $2 \leq r \leq 4$.

Let us first consider $2$-complexes that admit an embedding into $\R^2$. The $1$\mbox{-}skeleton of such a complex is a planar graph. We can use our knowledge on planar graphs to attain a bound in this case:
\begin{lem}
  A $2$-dimensional simplicial complex on $n \geq 3$ vertices that embeds into $\R^2$ can have at most $2n-5$ simplices of dimension $2$.
\end{lem}

\begin{proof}
Let $K$ be a $2$-dimensional simplicial complex that embeds into the plane. Let $G$ be the $1$-skeleton of $K$.
In the drawing of $G$ given by the embedding, the $2$-simplices of $K$ are among the inner faces of $G$.
Euler's formula together with the upper bound for the number of edges gives a maximum number of $2n-4$ (inner and outer) faces for a planar graph on $n$ vertices.
Hence, we get $f_2(K) \leq f(G) \leq 2n-5$.
\end{proof}

This bound is realized by any triangulation of the 2-simplex that has $n$ vertices and no additional vertices on the boundary. As there are complexes of this kind that embed linearly into the plane, in this case asking for linear or PL embeddings doesn't change the outcome of the question.

The following elementary treatment of the case $d=2$ and $\dR = 3$ can be found in \cite{Dey.1994}.
As remarked earlier, in this case the class of PL embeddable complexes equals the class of topologically embeddable ones.
\begin{lem}[{Dey, Edelsbrunner~\cite[(2.3)]{Dey.1994}}]\label{zweiDrei}
 A $2$-dimensional simplicial complex on $n\geq4$ vertices that embeds piecewise linearly into $\R^3$ can have at most $n(n-3)$
 $2$-simplices.
\end{lem}
% why only piecewise linear??? needed in the proof??
\begin{proof}
 Let $K$ be a $2$-dimensional simplicial complex on $n$ vertices with a piecewise linear embedding $f\!:\!||K|| \rightarrow \R^3$.

 For a vertex $v$ of $K$ we can estimate the number of adjacent triangles by considering a sphere centered at $f(v)$.
 If its radius is small enough, the intersection of the sphere with the triangles (and edges) adjacent to $v$ forms a planar graph.
 
 This graph cannot have more than $n-1\geq3$ vertices.
 Thus, by Theorem~\ref{sizePlanarGraphs} it has at most $3(n-1)-6$ edges.

 This means that every vertex in $K$ is adjacent to at most $3(n-1)-6$ triangles. 
%Taking care not to count every triangle three times we get
 Taking into account that we counted every triangle three times, we thus can see that $K$ has at most $\frac{1}{3} n (3(n-1)-6) = n(n-3)$ triangles.
\end{proof}
The $2$-skeleton of the cyclic $4$-polytopes show that this bound is tight. Again, as these are linearly embeddable complexes, in this case the answer to our question stays the same when asking for linear or PL embeddings.

%For $d=2$ and $\dR=4$, I don't know of any direct approach.
The case $d=2$ and $\dR=4$ is the smallest for which no direct approach seems to be known.
%This is one of the questions that will be addressed in the following chapters.
What can we say asymptotically about the maximal number of triangles among all $2$-complexes that embed in $\R^4$?
The complete complex yields the trivial upper bound of $O(n^3)$ in this case.

In Chapter \ref{lowerBounds} we will see that the $2$-skeleton of the cyclic $5$-polytope gives a lower bound of the  order $n^2$.
In Chapter \ref{upperBounds} we present a slightly better upper bound which is strictly less than~$n^3$, but which doesn't close the gap to the lower bound.
% Asymptotically the maximal number of triangles among all simplicial complexes that embed in $\R^4$ has to lie between $n^2$ and $n^3$.
% This is witnessed by the $2$-skeleton of the cyclic $4$-polytope and the complete complex.

\subsection*{3-Complexes}
We will only present a direct treatment of the case of $3$-complexes that admit a PL-embedding into $\R^3$.
The idea that yielded an upper bound for $2$-complexes in $\R^3$ can be used to deal with this case:
%For $3$-complexes  $\dR=3$.
%We can use the idea that yielded an upper bound for $2$-complexes in $\R^3$ to deal with the case of $3$-complexes that PL-embed into $\R^3$:

\begin{lem}
 A $3$-dimensional simplicial complex on $n\geq4$ vertices that embeds piecewise linearly into $\R^3$ can have at most $n(n-3)/2 - 1$ simplices of dimension~$3$.
\end{lem}

\begin{proof}
 Let $K$ be a $3$-complex with a piecewise linear embedding $f\!:\!||K|| \rightarrow \R^3$, and $v$ a vertex of $K$.
We can again consider a small enough sphere around the vertex $v$.
As above, the $2$-skeleton of $K$ yields a planar graph $G$ on the sphere.
Every tetrahedron of $K$ containing $v$ corresponds to a face of $G$.
Thus, $f_3(K) \leq n(2(n-1)-4)/4 = n(n-3)/2$.

To get a more precise bound, we can use the fact that there are at least four outer vertices of $K$,
i.e., vertices $v$ for which any neighborhood of $f(v)$ contains points not in $f(||K||)$.
For these vertices at least one face of the planar graph does not correspond to a tetrahedron of $K$.
So what we get is: 
\[
 f_3(K) \leq (n-4)(2(n-1)-4) + 4(2(n-1)-5) = n(n-3)/2 - 1.
\]
\end{proof}

This bound is attained by the boundary complex of the cyclic 4-polytope minus a facet, which admits a linear embedding into $\R^3$. For $d=\dR=3$ Problem~\ref{dieFrage} thus has the same answer for linear and for PL embeddable complexes.

\vspace{1em}
\noindent This concludes the discussion of the easily approachable cases. All results are summarized in Table \ref{Table}.
\begin{table}
\centering
$
\newcommand\T{\rule{0pt}{2.6ex}}
\newcommand\B{\rule[-1.4ex]{0pt}{0pt}}
%\renewcommand{\arraystretch}{1.4}
% use packages: array
\begin{array}{c||c|c|c|c}
 & d=1 \T \B & d=2 & d=3 & d=4\\ \hline\hline
\dR=1 \T \B & n-1 & 0 & 0 & 0\\ \hline
\dR=2 \T \B & 3n-6 & 2n-5 & 0 & 0\\ \hline
\dR=3 \T \B & \ueber{n}{2} & n(n-3) & \frac{n(n-3)}{2}-1^\ast& 0\\ \hline
\dR=4 \T \B & \ueber{n}{2} & ? & ? & ?\\ \hline
\dR=5 \T \B & \ueber{n}{2} & \ueber{n}{3} & ? & ?\\ \hline
\dR=6 \T \B & \ueber{n}{2} & \ueber{n}{3} & ? & ?\\ \hline
\dR=7 \T \B & \ueber{n}{2} & \ueber{n}{3} & \ueber{n}{4} & ?
\end{array}
$
\caption{The maximal number of $d$-simplices for a $d$-complex on $n$ vertices that embeds into $\R^{\dR}$ --- we will be looking for lower and upper bounds for the cells filled with ''{?}``.
\newline $^\ast$ {\footnotesize This holds for $3$-complexes that PL-embed into $\R^3$.}}\label{Table}
%\caption{$\max\{f_d(K)| \dim(K)=d, |V(K)|=n, ||K|| \hookrightarrow \R^{\dR} \}$ for low values of $d$ and $\dR$}
\end{table}

\subsection*{Linear embeddings of $d$-complexes in $\R^d$ and $\R^{d+1}$}
%We now state Day's and Peach's results on the cases $r=d$ and $r=d=1$ \cite{Dey.1998}.
In \cite{Dey.1998} Dey and Peach treat the cases $\dR=d$ and $\dR=d+1$ for linear embeddings with elementary methods. They prove more general statements; we now state their results in a reduced version.

\begin{thm}[{\cite[Theorem 2.1]{Dey.1998}}]\label{fulldim}
 One can select at most $O(n^{\lceil\frac{d}{2}\rceil})$ \linebreak \mbox{$d$-}dimensional simplices induced by $n$ points in $\R^d$ with the property that no~$2$ of them share a common interior point. This bound cannot be improved.
\end{thm}

\begin{thm}[{\cite[Theorem 3.1]{Dey.1998}}]\label{deInDePlusEins}
 Let $E$ be a family of $d$-simplices induced by an $n$-element point set $V\subseteq \R^{d+1}$ such that no $2$ members of $E$ have a common interior point. Then $|E|=O(n^d)$.
\end{thm}

As two maximal simplices of a simplicial complex are not allowed to share a common interior point, these upper bounds apply to the class of simplicial complexes with geometric realizations in $\R^d$, resp. $\R^{d+1}$.

The example showing the asymptotic tightness in the proof of Theorem \ref{fulldim} is a simplicial complex. Actually, it is the same example that will give us our lower bound in Section \ref{lowerBounds}. Hence, the tightness also holds for simplicial complexes:
\[
 \max\left\{f_d(K) \;\middle|\; \dim(K)=d,\, |V(K)|=n,\, K \text{ embeds linearly in } \R^d \right\} = \Theta(n^{\lceil\frac{d}{2}\rceil}).
\]
The bound in Theorem \ref{deInDePlusEins} is asymptotically tight for $d=2$, as we saw in Lemma \ref{zweiDrei} above. For $d\geq3$ we will see a lower bound of order $\Omega(n^{\lceil\frac{d+1}{2}\rceil})$.

\section{Algebraic Shifting and the Kalai-Sarkaria-Conjecture}\label{conjecture}
Before we address the questions of lower and upper bounds, we now present a conjecture by Kalai and Sarkaria on shifted complexes which implies a conjecture concerning Problem \ref{dieFrage} and also examine Conjecture \ref{Conj1} a little closer.

Let us first look at Conjecture \ref{Conj1}.
First, note that the bound on the number of triangles in this conjecture can be rephrased as a bound on the average number of triangles per edge:
Let $f_{12}(K)$ denote the number of edge-triangle incidences in a complex $K$, i.e., $f_{12}(K)=3f_2(K)$.
Then the condition $f_2(K) \geq 4 f_1(K)$ is equivalent to asking for 
\[
 \frac{f_{12}(K)}{f_1(K)} \geq 12.
\]

The $2$-skeleton of the boundary of the cyclic $5$-polytope shows that strengthening the conjecture by asking for $f_2(K) \geq c f_1(K)$ with $c<4$ is not possible:
For $K=(\partial C_5(n))^{\leq 2}$ we have that $K$ embeds into $\R^4$ and that
\[
 f_2(K)=4 \ueber{n}{2} -4n +10 \text{ and } f_1(K) = \ueber{n}{2}, 
\]
which shows that $\frac{f_2(K)}{f_1(k)} \geq c$ for every $c<4$ and $n$ large enough.

Originally, Kalai had also proposed a stronger version of this conjecture, which would apply to $K$ if it satisfies $f_2(K) \geq 4 f_1(K)-10 f_0(K)+20$.

The following example, found by Raman Sanyal and myself, contradicts the strengthening:
The boundary of a bipyramid over a triangle with the missing triangle added has $5$ vertices, $9$ edges and $7$ triangles and thereby fulfills the strengthened inequality. At the same time it is clearly embeddable into $\R^4$ and, indeed, into $\R^3$. (See Figure~\ref{counterexample}.)

\begin{figure}[htbp]
\centering
\includegraphics[scale=0.2]{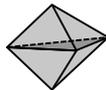}
\caption{The boundary of a bipyramid over a triangle with the missing triangle added.}\label{counterexample}
\end{figure}

% 
% Before we address the questions of lower and upper bounds, we now present a conjecture by Kalai and Sarkaria on shifted complexes which implies a conjecture concerning Problem \ref{dieFrage}.
A further conjecture concerning Problem \ref{dieFrage} can be derived from a conjecture by Kalai and Sarkaria on shifted complexes.
To phrase this statement we first need to learn a little bit about algebraic shifting, a prominent tool in $f$-vector theory. More details on and references to proofs for the results stated here can be found in \cite{Kalai.2002} and \cite{Bjoerner.1988}.

\begin{Def}[shifted complex]
 A simplicial complex $K$ on the vertex set $[n]$ is called \emph{shifted} if for any $F \in K$, $i \in F$ and $j < i$ we also have
 $(F\setminus\{i\}) \cup \{j\} \in K$. 
\end{Def}
This means in a shifted complex we can exchange any vertex of a simplex with a ``smaller'' one and still get a simplex of the complex.

There are several shifting operations which assign to a simplicial complex~$K$ a shifted complex $\Delta(K)$ with the same $f$-vector as $K$.
%Combinatorial shifting was introduced by Erd\"os, Ko and Rado \cite{Erdos.1961}.%and has been used in extremal set theory
Algebraic shifting, introduced by Kalai \cite{Kalai.1984,Kalai.1985}, is based on algebraic constructions. It preserves not only the $f$-vector but among other topological properties also the Betti-numbers (i.e., the ranks of the homology groups) of the complex.
Since every shifted complex is homotopy equivalent to a wedge of spheres, no shifting operation can preserve the homotopy type in general.

Algebraic shifting comes in two versions: exterior and symmetric algebraic shifting. For both operations, the resulting shifted complexes $\Delta^{ext}(K)$ and $\Delta^{symm}(K)$ of a simplicial complex $K$ depend on the characteristic of the field that is used in the construction.
These two operations do not coincide: For a simplicial complex $K$ the shifted complexes $\Delta^{ext}(K)$ and $\Delta^{symm}(K)$ are generally not the same.

The conjecture we are interested in involves $\mathcal{C}(\partial C_d(n))$, the boundary complex of the cyclic $d$-polytope $C_d(n)$, which we use here again before defining it in Section~\ref{CyclicPolytopes}.
For this complex, exterior and symmetric algebraic shifting coincide and the shifted complex $\Delta(\mathcal{C}(\partial C_d(n)))$ is the complex $\Delta(d,n)$ that is defined as follows (see \cite{Kalai.1991} for symmetric, \cite{Murai.2007} for exterior shifting):

\begin{Def}[$\Delta(d,n)$]
 The pure $(d-1)$-dimensional simplicial complex with vertex set $[n]$ and the set of maximal simplices
\[
 \left\{S \in \ueber{[n]}{d} \;\middle|\; k \notin S \Rightarrow [k+1,d-k+2] \subseteq S \text{ for all } k \in [n]\right\}
\]
is denoted by $\Delta(d,n)$.
\end{Def}

Kalai and Sarkaria independently stated the following conjecture: 

\begin{conj}[{Kalai, Sarkaria \cite[Conjecture~27]{Kalai.2002}}]\label{Conj2}
 If a simplicial complex $K$ on $n$ vertices admits an embedding into $S^{r}$, then $\Delta(K) \subseteq \Delta(r+1,n)$.
\end{conj}

The type of algebraic shifting, exterior or symmetric, is not specified in \cite{Kalai.2002}.
Because shifting preserves $f$-vectors, it follows from $\Delta(K) \subseteq \Delta(r+1,n)$ that $f_k(K) \leq f_k(C_{r+1}(n))$ for all $k \in \{0,\ldots,\dim(K)\}$.
Thus, the above conjecture would immediately imply the following:

\begin{conj}\label{Conj3}
Let $d,\dR \geq 1$ such that $d \leq \dR \leq 2d$ and let $n \geq \dR+1$.
Let furthermore $K$ be a $d$-complex on $n$ vertices that embeds into $\R^{\dR}$.

Then $K$ has at most as many $d$-simplices as $\partial C_{\dR+1}(n)$, the boundary complex of the cyclic ($\dR+1$)-polytope with $n$ vertices. Moreover, $f_k(K) \leq f_k(C_{\dR+1}(n))$ for all $k \in \{1,\ldots,d\}$.
Thus,
\[
 \max\left\{f_d(K) \;\middle|\; \dim(K)=d,\, |V(K)|=n,\, ||K|| \hookrightarrow \R^{\dR} \right\} \leq f_d(C_{\dR+1}(n))=O(n^{\lceil\frac{\dR}{2}\rceil}).
\]
\end{conj}
\vspace{1em}

%Note that Conjecture~\ref{Conj3} is in some sense independent of Conjecture~\ref{Conj1}, none of them follows immediately from the other:
Note that the combinatorial constraints on the complexes in Conjecture~\ref{Conj3} and Conjecture~\ref{Conj1} are not equivalent:
 For a $2$-dimen\-sional simplicial complex on $n$ vertices that admits an embedding into $\R^4$, Conjecture~\ref{Conj3} would yield
\[
 f_2(K) \leq f_2(C_5(n))=4 \ueber{n}{2} -4n +10 \text{ and } f_1(K) \leq f_1(C_5(n))= \ueber{n}{2}.
\]
This does neither directly imply nor directly follow from $f_2(K) < 4 f_1(K)$.

Conjecture~\ref{Conj2}, however, also implies Conjecture~\ref{Conj1}. To see this, we prove the following:
\begin{lem}\label{Delta5N}
 If $K\subseteq\Delta(5,n)$ is $2$-dimensional and shifted, then 
\[
 f_2(K) \leq 4f_1(K)-8 < 4 f_1(K).
\]
\end{lem}
\begin{proof}
We first determine the set of $2$-simplices of $\Delta(5,n)$.
A maximal simplex of $\Delta(5,n)$ is a set $S\subseteq[n]$ with $|S|=5$, satisfying the following constraints:
\begin{eqnarray*}
 1 \notin S &\implies& [2,6] \subseteq S,\\
 2 \notin S &\implies& [3,5] \subseteq S,\\
 3 \notin S &\implies& 4 \in S.
\end{eqnarray*}
The set of maximal simplices of $\Delta(5,n)$ is hence
\[
 \{23456\} \cup \left\{1345x \;\middle|\; x \geq 6\right\} \cup \left\{124xy \;\middle|\; 5 \leq x < y\right\} \cup \left\{123xy\;\middle|\; 4 \leq x < y\right\},
\]
where we write $abcde$ for the set $\{a,b,c,d,e\}$.
It follows from this that $\Delta(5,n)$ has the following $2$-simplices:
\[
 \left\{1xy \;\middle|\; 2 \leq x < y\right\} \cup \left\{2xy \;\middle|\; 3 \leq x < y\right\} \cup \left\{3xy \;\middle|\; 4 \leq x < y\right\} \cup \left\{4xy \;\middle|\; 5 \leq x < y\right\},
\]
where, again, $abc$ denotes the set $\{a,b,c\}$.

Now let $K$ be a shifted subcomplex of $\Delta(5,n)$. Every $2$-simplex of $K$ is of the form $mab$ with $1 \leq m \leq 4$ and $m < a < b$.
If $T$ denotes the set of $2$\mbox{-}simplices of $K$ and $E$ the set of edges, we consider the map $\varphi \!:\! T \to E$ that maps the $2$-simplex $mab$ to the edge $ab=\varphi(mab)$, which is in $K$ because $mab$ is in $K$.
A fixed edge in $K$ is the image of at most four $2$-simplices. Hence, $f_2(K) \leq 4 f_1(K)$.

Because $K$ is shifted and contains some $2$-simplex, the simplex $123$ and with it the edges $12$ and $13$ are in $K$.
The map $\varphi$ doesn't map any $2$-simplex to these two edges. This shows that
\[
 f_2(K) \leq 4(f_1(K)-2) < 4 f_1(K).
\]
\end{proof}
If Conjecture~\ref{Conj2} would hold, a $2$-complex $K$ admitting an embedding into $\R^4$ would satisfy $\Delta(K) \subseteq \Delta(5,n)$, and hence 
\[
 f_2(K) = f_2(\Delta(K)) < 4 f_1(\Delta(K)) = 4 f_1(K).
\]
%To see that Conjecture~\ref{Conj1} does not immediately imply Conjecture~\ref{Conj2}, consider the complex $K=\Delta(5,n)^{\leq 2}\cup\{567\}$.
To see that the combinatorial constraint on the complexes from Conjecture~\ref{Conj1} does not imply the statement of Conjecture~\ref{Conj2} consider the simplicial complex $K=\Delta(5,n)^{\leq 2}\cup\{567\}$.
This complex is shifted and fulfills
\[
 f_2(K) = f_2(\Delta(5,n)) +1 \leq 4f_1(\Delta(5,n))-7 < 4f_1(K),
\]
where the second inequality holds by Lemma~\ref{Delta5N} and the third because we have  $f_1(K)=f_1(\Delta(5,n))$.
\chapter{Lower Bounds}\label{lowerBounds}
%This chapter will be devoted to finding a lower bound for all instances of our problem.
%In some examples from the last chapter the $d$-skeleton of the cyclic $(\dR + 1)$-polytope proved the tightness of the given upper bound. This is exactly the example which will yield our lower bound in the general case.
%
% In this section, we will introduce polytopes in general and the aforementioned cyclic polytopes.
% We will see what polytopes are and how they can be used as examples of simplicial complexes.
% The cyclic polytopes will be shown to have maximum $f$-vectors among polytopes and will then yield a lower bound for all instances of our problem.
% The case $\dR = 2d$ will be considered afterwards.
This chapter will be devoted to finding lower bounds for all instances of Problem~\ref{dieFrage}.
We first introduce polytopes in general and, as an example, the aforementioned cyclic polytopes.
We will see that certain polytopes can be used as examples of embeddable complexes and that the cyclic polytopes have maximal $f$-vectors among them. These will then yield the promised lower bound.
%The case $\dR = 2d$ will be considered separately afterwards.
For the case $\dR = 2d$ we afterwards consider why this bound cannot be improved by simple means.
\section{Polytopes}\label{Polytopes}
%field of math?
Convex polytopes are geometric objects with a lot of underlying combinatorial structure. Spanning a whole area of mathematics, they are interesting on their own right and are not treated fairly when just considered as examples of simplicial complexes.
But as this is the focus of this thesis, we will have to make do with a very short introduction into the topic of convex polytopes, focusing on their combinatorial aspects, especially, of course, on the numbers of faces, and possible embeddings.
We will mainly follow \cite{Ziegler.1998}.

Let us begin with the definition of a convex polytope.
We will from now on just speak of ``polytopes'' as we will not consider non-convex polytopes.
%etw. gerecht werden?
%proofs for all things in Ziegler? GRuenbaum?
%
\begin{Def}[polytope]
 A \emph{polytope} $P$ in $\R^{\dR}$ is the convex hull of finitely many points in $\R^{\dR}$.
  The \emph{dimension} of $P$ is the dimension of its affine hull.
 If $P$ is \mbox{$d$-dimensional} it is called a \emph{$d$-polytope}.
\end{Def}

This tells us that we already saw examples of polytopes: $d$-simplices, which are convex hulls of $d+1$ points in $\R^{\dR}$ for $r \geq d$.
The $d$-cube is another elementary example of a polytope. It is the convex hull of all points with $0/1$-coordinates in $\R^d$.
Further examples are the cyclic polytopes which we will consider in the next section.
%We will encounter another example in the next section: Cyclic polytopes.

A major basic theorem in polytope theory states that a polytope can equivalently be defined as a bounded intersection of a finite family of closed halfspaces in some $\R^{\dR}$ (e.g., \cite[Theorem 1.1]{Ziegler.1998}).
Here, ``bounded'' means not containing any ray of the form $\left\{x + \lambda y \;\middle|\; \lambda \geq 0\right\}$.

We will now define
the faces of a polytope which connect the geometric object with a combinatorial structure.

\begin{Def}[face of a polytope]
 Let $P$ be a polytope in $\R^{\dR}$.
 If $H$ is a hyperplane for which $P$ lies entirely in one of the halfspaces determined by $H$, the intersection of $P$ and $H$ is called a \emph{face} of $P$.
 The intersection of $P$ with the entire space $\R^{\dR}$, i.e., $P$ itself, is also considered as a face of $P$.
 All other faces are referred to as \emph{proper faces}.
\end{Def}

For every polytope it is possible to find a hyperplane that doesn't intersect~$P$. Thus, the empty set is a face of every polytope.

One can also quickly see that every face $F$ of $P$ is itself a polytope: Let $H_F$ be a hyperplane defining $F$. Then $F$ is the intersection of $P$ with the halfspace determined by $H_F$ that does not contain $P$ entirely and hence an intersection of finitely many halfspaces.

Thus, we have a notion of dimension for the faces of a polytope.
Faces of dimension $0$ are referred to as \emph{vertices}.
A $(d-1)$-dimensional face of a $d$\mbox{-}polytope is called a \emph{facet}.

One can also show the following basic properties of faces of polytopes:
\begin{prop}[{\cite[Proposition 2.2, Proposition 2.3 and Exercise 2.4]{Ziegler.1998}}]\label{faces}
Let $P \subset \R^{\dR}$ be a polytope.
\begin{itemize}
 \item Let $V$ be the set of vertices of $P$. Then $P=\conv(V)$.
 \item If $X \subset \R^{\dR}$ is finite and $P=\conv(X)$, then $X$ contains the vertices of $P$. 
 \item Every proper face of $P$ is contained in a facet of $P$.
 \item For a face $F$ of $P$, every face of $F$ is also a face of $P$.
 \item The intersection of two faces of $P$ is again a face of $P$.
\end{itemize}
\end{prop}

Let $P$ be a polytope for which all proper faces are simplices. The last two statements in Proposition \ref{faces} show that the set of all proper faces of $P$ is then a simplicial complex.
Polytopes with this property are called simplicial:

\begin{Def}[simplicial polytope]
 A polytope $P$ is \emph{simplicial} if every facet of $P$ is a simplex. %, and therefore also every proper face of $P$,
\end{Def}

Note that, as all faces of a simplex are again simplices, by Proposition \ref{faces} it is enough to ask for the facets of $P$ to be simplices. See Figure \ref{polytopeFigure} for an example of a simplicial and a non-simplicial polytope.

\begin{figure}[ht]
\begin{center}
\includegraphics[scale=0.5]{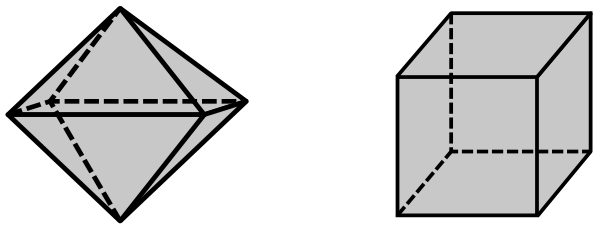}
\end{center}
\caption{An octahedron and a $3$-cube as examples of a simplicial and a non-simplicial polytope.}\label{polytopeFigure}
\end{figure}

We will be interested in simplicial polytopes with an additional property, neighborliness. The Upper Bound Theorem, which we will consider in more detail in Section \ref{CyclicPolytopes}, states that these have maximal numbers of faces.

\begin{Def}[neighborly polytope]
 A polytope $P$ is \emph{$k$-neighborly} if any subset of at most $k$ of its vertices is the vertex set of a face of $P$.
A $\lfloor\frac{d}{2}\rfloor$-neighborly $d$-polytope is also just called \emph{neighborly}.
\end{Def}
The $d$-simplex, e.g., is $d$-neighborly.
It is not hard to see that any other polytope can be at most $\lfloor\frac{d}{2}\rfloor$-neighborly (e.g., \cite[Section 7.1]{Gruenbaum.2003}).
The cyclic polytopes will turn out to be examples of neighborly simplicial polytopes.

%What can we say about the embeddability of the simplicial complexes we get from simplicial polytopes?
%Of course, every $d$-polytope can be embedded into $R^d$, but for the boundary of the polytope we can gain one dimension.

%To formalize this statement, we will need polytopal complexes, a generalization of simplicial complexes:
%
The set of proper faces of a polytope always forms a sort of complex; it might just not consist of simplices, but of other polytopes. There is a generalization of the concept of simplicial complexes capturing this situation:
\begin{Def}[polytopal complex]
 A \emph{polytopal complex} in $\R^\dR$ is a non-empty family $\mathcal{C}$ of polytopes in $\R^\dR$ with:
 \begin{enumerate}
  \item 
   If $P \in \mathcal{C}$ and $P'$ is a face of $P$, then $P' \in \mathcal{C}$.
   \item
    $P_1, P_2 \in \mathcal{C} \Rightarrow P_1 \cap P_2$ is a face of both $P_1$ and $P_2$.
 \end{enumerate}
The \emph{dimension} of a polytopal complex $\mathcal{C}$ is $\dim(\mathcal{C}) \mathrel{\mathop:}= \max\left\{\dim(P)\;\middle|\;P \in \mathcal{C}\right\}$.
The set $|\mathcal{C}| \mathrel{\mathop:}= \bigcup_{P \in \mathcal{C}}P$ is called the \emph{underlying set} of $\mathcal{C}$.
\end{Def}

With a polytope $P$ we can associate $\mathcal{C}(\partial P)$, the \emph{boundary complex} of $P$, which consists of all proper faces of $P$ and by Proposition \ref{faces} is a polytopal complex.
Of course $P$ itself can also be considered as a polytopal complex.
As remarked earlier, for simplicial polytopes the boundary complex is a simplicial complex.

% We are again interested in the numbers of faces of a polytopal complex.
% So, just as for simplicial complexes we define the $f$-vector of a polytopal complex:
% 
% \begin{Def}The \emph{$f$-vector} of a $d$-dimensional polytopal complex $\mathcal{C}$ is the vector
% $f(\mathcal{C}) = (f_0(\mathcal{C}),f_1(\mathcal{C}),\ldots,f_{d}(\mathcal{C}))$ where $f_k(\mathcal{C})$ denotes the number of $k$\mbox{-}faces in $\mathcal{C}$.
% 
% \end{Def}
% Note that for some purposes it is convenient to add $f_{-1}(\mathcal{C})=1$, corresponding to the empty face, as first entry of the $f$-vector. For our intentions this is not necessary.
% The \emph{$f$-vector} of a polytope $P$ is the $f$-vector of its boundary complex:
% $f(P) = f(\mathcal{C}(\partial P)) = (f_0(P),f_1(P),\ldots,f_{d-1}(P))$ where $f_k(P)$ denotes the number of $k$\mbox{-}faces of $P$.

%We need a concept of equivalence for polytopal complexes:
To be able to compare and equate polytopal complexes we need a concept of equivalence:
\begin{Def}[combinatorial equivalence]
 Two polytopal complexes $\mathcal{C}_1, \mathcal{C}_2$ are \emph{combinatorially equivalent} if there exists a bijective map $\phi:\mathcal{C}_1 \to \mathcal{C}_2$ such that $P_1' \subseteq P_1$ if and only if $\phi(P_1') \subseteq \phi(P_1)$ for $P_1', P_1 \in \mathcal{C}_1$.
\end{Def}
%
%subcomplex, maps, affinely isomorphic?
%%
Note that this notion is a generalization of the concept of isomorphisms for simplicial complexes: A bijective simplicial map with simplicial inverse induces a map like this on the simplicial complexes.
%: A bijective simplicial map with simplicial inverse induces a poset isomorphism on the face posets corresponding to the simplicial complexes involved.

%For a polytope $P$ the face poset of the corresponding complex satisfies the axioms of a lattice and is hence called the \emph{face lattice} $L(P)$ of $P$. Two polytopes are called $P_1$ and $P_2$ are \emph{combinatorially equivalent} if the corresponding polytopal complexes are combinatorially equivalent: $L(P_1) \cong L(P_2)$ (as posets).
%
%lattice?
By interpreting a polytope as a polytopal complex consisting of all of its faces, one gains from this definition a notion of combinatorial equivalence for polytopes.
For polytopes (and also for polytopal complexes) there is also a second, geometric notion of equivalence:
Two polytopes $P_1$ and $P_2$ in $\R^{\dR_1}$ and $\R^{\dR_2}$ respectively are \emph{affinely isomorphic} if there is an affine map $f\!:\!\R^{\dR_1} \rightarrow \R^{\dR_2}$ such that $f|_{P_1}$ is bijective onto $P_2$.

Observe that this is a stronger notion: While the existence of an affine isomorphism implies combinatorial equivalence, two combinatorially equivalent polytopes do not have to be affinely isomorphic. Any affine isomorphism maps, e.g., the square $\conv((0,0),(0,1),(1,0),(1,1))$, which is combinatorially equivalent to any $2$-polytope with $4$ vertices, onto a parallelogram.
As we are more interested in the combinatorial aspects of polytopes, we will use the combinatorial concept.

%The face poset of the boundary of a polytope is denoted by $L(\partial P) := L(\mathcal{C}(\partial P)) = L(P)\setminus P$.
%
% A \emph{polytopal subdivision} of a polytope $P$ is a polytopal complex $\mathcal{C}$ with $|\mathcal{C}| = P$.
% 
%
%polytopal subdivision is generalization of triangulation.
%
%
%In the remainder of this section
Before we move on to the introduction of cyclic polytopes we will now consider the embeddability of the simplicial complexes that arise from simplicial polytopes.
% While every $d$-polytope of course embeds into $R^d$, we will now see that most of the boundary of a $d$-polytope (everything apart from one facet) can even be embedded into $\R^{d-1}$.

%Let $P$ be a $d$-polytope in $R^d$.
Let $P$ be a $d$-polytope. Like every $d$-polytope, $P$ can be embedded into $\R^d$.
By projecting from an inner point of $P$ onto a surrounding sphere, one can see that the underlying set of $\mathcal{C}(\partial P)$ is homeomorphic to a $(d-1)$-sphere. Thus, if we take out a facet $F$ (even one point would be enough, but we want to maintain the structure of a complex), the stereographic projection maps $\mathcal{C}(\partial P)\setminus F$ homeomorphically onto some subset of $\R^{d-1}$.

Another way to see this, which even yields a linear embedding, is via the technique of \emph{Schlegel diagrams} (e.g., \cite[Chapter 5]{Ziegler.1998}).
This gives us the following result:
%We can embed most of the boundary of a $d$-polytope (everything apart from one facet) into $\R^{d-1}$:
%this is a straight-line embedding!
\begin{thm}[{e.g., \cite[Proposition 5.6]{Ziegler.1998}}]\label{schlegel}
 For every $d$-polytope $P$ and every facet $F$ of $P$ there exists a polytopal complex in $\R^{d-1}$ that is combinatorially equivalent to $\mathcal{C}(\partial P)\setminus\{F\}$.
\end{thm}

In summary, the main result of this section is that we can get examples for embeddable complexes from simplicial polytopes:
For a simplicial $d$-polytope $P$ and every facet $F$ of $P$ we have the simplicial complex $\mathcal{C}(\partial P)\setminus F$.
This complex is isomorphic to a (geometric) simplicial complex in $\R^{d-1}$ and thus admits a linear embedding into $\R^{d-1}$.
%need: only (d-1)-polytope with d vertices is the simplex, every facet simplex => simplicial
%face which is not a facet is face of a facet
\section{The Cyclic Polytopes and the Upper Bound Theorem}\label{CyclicPolytopes}

In this section, having mentioned them several times already, we finally introduce cyclic polytopes.
%As already mentioned several times, in this section we present our most import examples of polytopes, the cyclic polytopes.
After seeing their definition and some of their properties, we will encounter the promised Upper Bound Theorem which states that for every $k$, the cyclic $d$-polytope on $n$ vertices has the maximal number of $k$-faces among all $d$-dimensional polytopes $P$ with $f_0(P)=n$.

To define cyclic polytopes we need to recall the definition of the moment curve, introduced in Section \ref{embeddings}.
For $d \geq 2$ and $n>d$, let $\gamma\!:\!\R \rightarrow \R^d$ be the moment curve in $\R^d$, i.e., $\gamma(t) = (t,t^2,\ldots,t^d)$, and choose $t_1 < t_2 < \ldots < t_n \in \R$.
What can we say about the polytope $P\mathrel{\mathop:}=\conv(\{\gamma(t_1),\gamma(t_2),\ldots,\gamma(t_n)\})$?

We saw in Lemma \ref{momentcurve} that
no $d+1$ distinct points on $\gamma$ can lie in a common hyperplane.
This immediately tells us two things:
Firstly, the affine hull of $P$ must be all of $\R^d$ because $n>d$. This makes $P$ a $d$-dimensional polytope.

Furthermore, as a facet is the intersection of $P$ with a hyperplane, no facet of $P$ can contain more than $d$ points of $\{\gamma(t_1),\gamma(t_2),\ldots,\gamma(t_n)\}$.
A facet, as a $(d-1)$-polytope, has to have at least $d$ vertices which by Proposition \ref{faces} have to be among the $\gamma(t_i)$. Thus, every facet of $P$ has $d$ affinely independent vertices and is hence a simplex.
This shows that $P$ is a simplicial $d$-polytope.
With a little more effort one can say much more about the combinatorial properties:

\begin{thm}[{Gale's Evenness Condition, e.g., \cite[Theorem 0.7]{Ziegler.1998}}]
Let $d \geq 2$, $n>d$, choose $t_1 < t_2 < \ldots < t_n \in \R$ and set 
\[
 C_d(t_1,t_2,\ldots,t_n) \mathrel{\mathop:}= \conv(\{\gamma(t_1),\gamma(t_2),\ldots,\gamma(t_n)\}).
\]
 Identify $\gamma(t_i)$ with $i$ and choose $S \subseteq [n]$ with $|S|=d$.
 Then the following two statements are equivalent:
 \begin{enumerate}
 \item $S$ is the vertex set of a facet of $C_d(t_1,t_2,\ldots,t_n)$.
 \item For all $i,j\notin S$ with $i<j$ the number of $k \in S$ between $i$ and $j$ is even. 
 \end{enumerate}
\end{thm}
This tells us for example that the vertex set of $P$ is $\{\gamma(t_1),\gamma(t_2),\ldots,\gamma(t_n)\}$:
For every $1 \leq i \leq n$ one can find a set $S \in \ueber{[n]}{d-1}$ such that $S \cup \{i\}$ satisfies the evenness condition. Thus, every $\gamma(t_i)$ is contained in a vertex set of a facet of $P$ and is hence a vertex of $P$.

%Actually this shows much more:
%Independently of the choice of parameters $t_1,\ldots,t_n$, Gale's Evenness Condition determines the complete combinatorics of $C_d(t_1,t_2,\ldots,t_n)$, which justifies the following definition:
Actually, Gale's Evenness Condition determines the complete combinatorics of $C_d(t_1,t_2,\ldots,t_n)$, independently of the choice of parameters $t_1,t_2,\ldots,t_n$, which justifies the following definition:
\begin{Def}[cyclic polytope]
 For $d \geq 2$ and $n > d$, the notation $C_d(n)$ refers to any member of the combinatorial equivalence class
 of $C_d(t_1,t_2,\ldots,t_n)$ for some choice of parameters $t_1 < t_2 < \ldots < t_n \in \R$.
We call $C_d(n)$ the \emph{cyclic $d$-polytope} on $n$ vertices.
 Its vertex set is identified with $[n]$ by $\gamma(t_i) \mapsto i$.
\end{Def}
The evenness condition can easily be extended to describe all proper faces of $C_d(n)$ (e.g., \cite[Exercise 0.8]{Ziegler.1998}), which then shows that the cyclic polytope $C_d(n)$ is neighborly, i.e., that any subset of at most $\lfloor\frac{d}{2}\rfloor$ vertices forms the vertex set of a face.

%proof? evenness condition for all faces?
%
%Why are we interested in cyclic polytopes?

We already mentioned before that, for all $k$, the cyclic polytope $C_d(n)$ has the maximal number of $k$-faces among all $d$-polytopes
on $n$ vertices.
This result is known as the Upper Bound Theorem for polytopes.
For us this means that cyclic polytopes will yield the best examples we can obtain from polytopes.

To be able to phrase the theorem more compactly, we first introduce a notation for the numbers of faces of a polytope. Just as for simplicial complexes this is the $f$-vector of a polytope:

\begin{Def}[$f$-vector of a polytope]
The \emph{$f$-vector} of a $d$-dimensional polytope $P$ is the vector
$f(P)  = (f_0(P),f_1(P),\ldots,f_{d-1}(P))$ where $f_k(P)$ denotes the number of $k$\mbox{-}faces of $P$. 
\end{Def}
Note that for some purposes it is convenient to add $f_{-1}(\mathcal{P})=1$, corresponding to the empty face, as first entry of the $f$-vector. For our intentions this is not necessary.
\begin{thm}[{Upper Bound Theorem, e.g., \cite[Theorem 8.23]{Ziegler.1998}}]
\label{UBT}
 Let $P$ be a $d$-polytope on $n$ vertices and let $k < d$.
 Then $P$ can have at most as many $k$-faces as the cyclic polytope $C_d(n)$: $f_k(P)\leq f_k(C_d(n))$.

 If $f_k(P) = f_k(C_d(n))$ for some $k$ with $\lfloor\frac{d}{2}\rfloor-1 \leq k \leq d-1$, then $P$ is neighborly.
\end{thm}

The first complete proof of the Upper Bound Theorem for all polytopes was given by McMullen \cite{McMullen.1970}.
McMullen's proof, also given in \cite{Ziegler.1998}, which involves deeper results of polytope theory than we have at hands, also shows that all simplicial neighborly $d$-polytopes on $n$ vertices have the same $f$-vector as $C_d(n)$ and yields the exact number of faces:
\begin{prop}[{e.g., \cite[Corollary 8.28]{Ziegler.1998}}]\label{simpNeighb}
 Let $P$ be a simplicial neighborly $d$-polytope with $n$ vertices and let $k < d$.

 Then $f_k(P) = \sideset{}{^*}\sum_{i=0}^{\frac{d}{2}} \left( \ueber{d-i}{k+1-i} + \ueber{i}{k+1-d+i} \right) \ueber{n-d-1+i}{i}$, where
  \[
   \sideset{}{^*}\sum_{i=0}^{\frac{d}{2}} T_i = \begin{cases} 
                                        T_0 + T_1 + \ldots + T_{\lfloor\frac{d}{2}\rfloor} & \text{if } d \text{ is odd}, \\
                                        T_0 + T_1 + \ldots + T_{\lfloor\frac{d}{2}\rfloor-1} + \frac{1}{2} 
                                        T_{\lfloor\frac{d}{2}\rfloor} & \text{if } d \text{ is even}.
	                              \end{cases}
 \]
\end{prop}
%
%The Upper Bound Theorem does not only hold for polytopes, but also for several classes of simplicial complexes, see \cite{Novik.2003} for a list of these classes.
That the Upper Bound Theorem holds is known not only for polytopes, but also for several classes of simplicial complexes, see, e.g., \cite{Novik.2003} for a list of such classes.
Stanley proved that it holds for all simplicial spheres, i.e., simplicial complexes that have a triangulated sphere as geometric realization (\cite{Stanley.1975}, also see chapter II.3 in \cite{Stanley.1996}).
This means that
\[
 f_k(K) \leq f_k(C_d(n)) \text{ for } k=0,1,\ldots,d-1
\]
if $K$ is a simplicial complex with $||K|| \cong S^{d-1}$ and $f_0(K)=n$.
This really is a stronger statement than Theorem \ref{UBT} as not all triangulations of spheres are polytopal; actually most are not \cite{Kalai.1988}.
\section{A Lower Bound via Simplicial Polytopes}\label{simplicialpolytopes}
The last sections showed that simplicial polytopes yield examples of simplicial complexes which are embeddable into a certain Euclidean space and that the neighborly ones among them, in particular the cyclic polytopes, will give maximal $f$-vectors.
We will now apply these insights to Problem~\ref{dieFrage}. Recall that we are interested in the maximal number of $d$-simplices for a $d$-dimensional simplicial complex on $n$ vertices that embeds into $\R^{\dR}$ for some $d\leq \dR \leq 2d$.
Theorem~\ref{schlegel} and Proposition~\ref{simpNeighb} yield the following lower bound for this:

\begin{prop}\label{theLowerBound}
 Let $d, \dR \in \mathbb{N}$ such that $d \leq \dR \leq 2d$ and let $n \geq r+1$.
Then 
\[
  \max\left\{f_d(K)\;\middle|\; \dim(K)=d, |V(K)|=n, ||K|| \hookrightarrow \R^{\dR} \right\} \geq f_d(C_{\dR + 1}(n)) = \Omega(n^{\lceil\frac{\dR}{2}\rceil}).
 \]
These bounds can be attained by simplicial complexes that admit a linear embedding into $\R^{\dR}$.
For the case $d=r$, the same asymptotic lower bound also holds for embeddability into $S^d$:
\[
  \max\left\{f_d(K)\;\middle|\; \dim(K)=d, |V(K)|=n, ||K|| \hookrightarrow S^d \right\} \geq f_d(C_{d + 1}(n))= \Omega(n^{\lceil\frac{d}{2}\rceil}).
 \]
\end{prop}

\begin{proof}
 First, consider the case $\dR > d$.
 The results of the last two sections show that the complex $K = \mathcal{C}(\partial C_{\dR + 1}(n))^{\leq d}$ is a $d$-dimensional simplicial complex
 on $n$ vertices which, by Theorem~\ref{schlegel}, admits a linear embedding into $\R^{\dR}$.

 By Proposition \ref{simpNeighb} the number of $d$-simplices of $K$ is
 \[
 f_d(K) = \sideset{}{^*}\sum_{i=0}^{\frac{\dR + 1}{2}} \left( \ueber{\dR + 1-i}{d+1-i} + \ueber{i}{d-\dR +i} \right)
 \ueber{n -\dR -2 +i}{i} = \Omega(n^{\lceil\frac{\dR}{2}\rceil}).
 \]
 For the case $r=d$ consider $K = \mathcal{C}(\partial C_{d + 1}(n)) \setminus F$ for some facet $F$ of $C_{d + 1}(n)$.
 $K$ embeds linearly into $\R^d$ and has 
 \[
  f_d(K) = \sideset{}{^*}\sum_{i=0}^{\frac{d + 1}{2}} 2 \ueber{n -d -2 +i}{i} - 1 = \Omega(n^{\lceil\frac{d}{2}\rceil}) 
 \]
 $d$-simplices.
The complex $\mathcal{C}(\partial C_{d + 1}(n))$ embeds into $S^d$ and has one more simplex.
\end{proof}

Note that Conjecture \ref{Conj3} together with Proposition \ref{theLowerBound} would imply that $f_d(C_{\dR + 1}(n))$ is the solution to Problem \ref{dieFrage}.
As the complexes yielding the lower bound in Proposition \ref{theLowerBound} embed linearly and any upper bound for the problem on topological embeddability also applies for the linear version of the question, this would also mean that the answer to Problem \ref{dieFrage} doesn't change when considered for linear or PL embeddable complexes.

In Section \ref{size} we considered Problem \ref{dieFrage} for several low values of $d$ and~$\dR$. For $\dR=3$ and $d\in\{2,3\}$, we already referred to the cyclic polytope $C_4(n)$ as an example attaining the presented upper bounds. Now, we can see that it actually has
\begin{eqnarray*}
 f_2(C_4(n)) & = & \sideset{}{^*}\sum_{i=0}^{\frac{4}{2}} \left( \ueber{4-i}{3-i} + \ueber{i}{i-1} \right) \ueber{n-5+i}{i}\\
& = & \ueber{4}{3} + \left(\ueber{3}{2} + \ueber{1}{0}\right) \ueber{n-4}{1} + \frac{1}{2} \left( \ueber{2}{1} + \ueber{2}{1} \right) \ueber{n-3}{2}\\
& = & 4 + 4 (n-4) + (n-3)(n-4)\\
& = & n(n-3)
\end{eqnarray*}
faces of dimension $2$ and
\begin{eqnarray*}
 f_3((C_4(n)) & = & \sideset{}{^*}\sum_{i=0}^{\frac{4}{2}} 2
 \ueber{n -5 +i}{i}\\
& = & 2 \left( \ueber{n-5}{0} + \ueber{n-4}{1} + \frac{1}{2} \ueber{n-3}{2}\right)\\
& = & \frac{4 + 4(n-4) + (n-3)(n-4)}{2}\\
& = & \frac{n(n-3)}{2}
\end{eqnarray*}
$3$-simplices.

As remarked, the case $d=2$ and $\dR = 4$ seems to be the simplest that is yet unsolved. Here we consider the complex $\mathcal{C}(\partial C_{5}(n))^{\leq 2}$ for which we can calculate in a similar fashion that it has $2(n^2 - 6n + 10) = 4 \ueber{n}{2} -4n +10$ triangles, a fact we already used when discussing the different conjectures concerning Problem \ref{dieFrage}.
For this case, we thus get a lower bound of $\Omega(n^2)$. The trivial upper bound given by the complete complex in this case is $O(n^3)$.
In the forthcoming sections we will see that it is not possible to push up the lower bound by adding further \mbox{$2$-simplices} to $\mathcal{C}(\partial C_{5}(n))^{\leq 2}$ while preserving embeddability into $\R^4$. Furthermore, the chapter on upper bounds will present an upper bound of $O(n^{3-\frac{1}{9}})$ which slightly improves the trivial bound but is still far from this lower bound.

For the more general case $\dR=2d$ the complex $\mathcal{C}(\partial C_{2d + 1}(n))^{\leq d}$ yields a lower bound of $\Omega(n^d)$.
For all these cases we will show in the remainder of this part on lower bounds that the examples given by neighborly simplicial polytopes cannot be improved in the fashion described above.

The Upper Bound Theorem shows that this is the best lower bound we can achieve by considering simplicial polytopes, and also simplicial spheres in general, as examples of embeddable complexes.
\section{Improving the Lower Bound (for the Case \texorpdfstring{$\dR=2d$}{})?}\label{improveLowerBound}

In the previous section we attained a lower bound for the maximal number of faces of a complex that is embeddable into a certain Euclidean space by presenting the boundary complex of the cyclic polytope as an example of such a complex.
Now we want to study whether this bound can be improved by adding further simplices to the complex $\mathcal{C}(\partial C_{\dR + 1}(n))^{\leq d}$ without adding new vertices and such that it still stays embeddable into $\R^{\dR}$.

For the cases $\dR = 2d$ we will show that this is not possible. First, we will present a theorem that considers this situation in a more general context and gives the negative answer for our case. Then we will try to see more directly why the embeddability is destroyed when adding an additional simplex, by looking for subcomplexes and minors of which we know that they are not embeddable.

%For the cases where $\dR < 2d$, we are not able to decide this question.
The methods we will you use do not apply for the cases where $\dR < 2d$, hence the question remains to be answered for these cases.

%\subsection{A theorem by Nevo and Wagner}
\subsection{Adding Missing Faces to Skeleta of PL-Spheres}

%In the case $\dR = 2d$ one way of improving these examples could be to add further $d$-simplices to the complex $\mathcal{C}(\partial C_{2d + 1}(n))^{\leq d}$ while conserving embeddability into $\R^{2d}$.

%This is not possible by the following theorem by Nevo and Wagner:
Consider the following theorem by Nevo and Wagner:

\begin{thm}[{Nevo, Wagner~\cite[Theorem 1.2]{Nevo.2008}}]\label{NevoWagnerThm}
 Let $S$ be a piecewise linear $2d$-sphere and $M$ a missing $d$-face of $S$.
 Then $S^{\leq d} \cup \{M\}$ does not embed into $S^{2d}$.
\end{thm}

To see why this theorem captures our situation, we first have to explain some of the terms that appear in the theorem.

\begin{itemize}
 \item Let $M \subseteq V(K)$ be a $(k+1)$-subset of the vertices of a simplicial complex $K$.
We call $M$ a \emph{missing $k$-face} of $K$ if $M$ is not contained in $K$ but its boundary is: $M \notin K$, $\partial M \subseteq K$.

 \item Two simplicial complexes are \emph{piecewise linearly homeomorphic (PL-homeo\-morphic)}  if they have subdivisions that are isomorphic as simplicial complexes.
A \emph{piecewise linear $m$-sphere (PL-$m$-sphere)} is then a simplicial complex $S$ that is piecewise linearly homeomorphic to $\partial\Delta^{m+1}$.

A \mbox{PL-$m$-sphere} $S$ thus has a subdivision that is isomorphic to some subdivision of $\partial\Delta^{m+1}$.
(See \cite{Lickorish.1999} for an introduction to piecewise linear topology.)
\end{itemize}

Let us now check whether we have the right ingredients for Theorem \ref{NevoWagnerThm}:
The complex $\mathcal{C}(\partial C_{2d + 1}(n))$ is a piecewise linear $2d$-sphere because any simplicial subdivision of the boundary of an $m$-polytope is PL-homeomorphic to $\partial\Delta^{m}$ (e.g., \cite[Lemma 4.2]{Lickorish.1999}) and thus a PL-$(m-1)$-sphere.

Furthermore $C_{2d + 1}(n)$ is $d$-neighborly, which means that its boundary complex contains all possible $(d-1)$-simplices.
So any $(d+1)$-subset of $[n]$ that is not the vertex set of a $d$-simplex of $\mathcal{C}(\partial C_{2d + 1}(n))$ is a missing $d$-face.

And, as mentioned before, for a $d$-complex embeddability into $S^{2d}$ is equivalent to embeddability into $\R^{2d}$.
Therefore, this theorem yields that the complex $\mathcal{C}(\partial C_{2d + 1}(n))^{\leq d} \cup \{M\}$ does not admit an embedding
into $\R^{2d}$ for any $(d+1)$-set $M \subseteq [n]$ that is not the vertex set of a $d$-face of $C_{2d + 1}(n)$.

This reasoning obviously works for any simplicial neighborly polytope:
\begin{cor}
 Let $P$ be a simplicial neighborly $(2d+1)$-polytope and let $M$ be a set of $d+1$ vertices of $P$
 that does not correspond to a $d$-face of $P$.
 Then $\mathcal{C}(\partial P)^{\leq d} \cup \{M\}$ does not embed into $\R^{2d}$.
\end{cor}
Of course this doesn't mean that one could not get a bigger complex than $\mathcal{C}(\partial C_{2d + 1}(n))^{\leq d}$ by starting off
with a non-neighborly simplicial polytope (or simplicial sphere) with less $d$-faces and adding ``non-missing'' non-faces to it.

Since not all simplicial spheres are piecewise linear (e.g., \cite{Lickorish.1999}, p.302), the most general statement we can achieve for our purposes with similar reasoning is the following:
\begin{cor}
Let $K$ be a $d$-dimensional simplicial complex on $n$ vertices that admits an embedding into $\R^{2d}$ and assume, furthermore, that there exists a subcomplex $K'\subseteq K$ with $V(K')=V(K)$ such that $K'=S^{\leq d}$ for some neighborly PL-$2d$-sphere $S$.
Then $K=K'=S^{\leq d}$ and $f_d(K) \leq f_d(C_{2d + 1}(n))$.
\end{cor}
% \begin{cor}
% Let $S$ be a neighborly piecewise linear $2d$-sphere.
% Without adding new vertices, its $d$-skeleton $S^{\leq d}$ cannot be extended to a supercomplex that still admits an embedding into $\R^{2d}$.
% 
% Phrased differently: If $K$ is a $d$-dimensional simplicial complex on the vertices of $S$ that embeds into $\R^{2d}$ and has $S^{\leq d}$ as a subcomplex
% then $K=S^{\leq d}$ and $f_d(K) \leq f_d(C_{2d + 1}(n))$.
% \end{cor}
%
Note that neighborliness for simplicial complexes has exactly the same meaning as for polytopes.

In \cite{Nevo.2008} Nevo and Wagner conjecture that Theorem \ref{NevoWagnerThm} holds for general simplicial spheres. 
For our question, this would mean that no example coming from a neighborly simplicial sphere, also not the ones coming from non-\mbox{PL-spheres}, can be extended to yield a higher bound than $f_d(C_{2d + 1}(n))$.

For the other cases, when $\dR < 2d$, Theorem \ref{NevoWagnerThm} doesn't answer our question:
The complex $\mathcal{C}(\partial C_{\dR + 1}(n))$ is a piecewise linear $\dR$-sphere, but to apply Theorem~\ref{NevoWagnerThm} we first of all need $\dR$ to be even.
Furthermore we have to add some missing $\frac{\dR}{2}$-face by adding the additional $d$-simplex.

For even $\dR$, the complex $\mathcal{C}(\partial C_{\dR + 1}(n))$ is $\frac{\dR}{2}$-neighborly, which means that any $(\frac{\dR}{2}+1)$-subset of $[n]$ that is not the vertex set of a $\frac{\dR}{2}$-simplex is a missing $\frac{\dR}{2}$-face.
Thus, if the boundary of the $d$-simplex that we want to add to $\mathcal{C}(\partial C_{\dR + 1}(n))^{\leq d}$ does not contain all possible $\frac{\dR}{2}$-simplices, adding the $d$-simplex would also mean adding a missing $\frac{r}{2}$-face.
% Only if $\dR$ is even and if the boundary of the $d$-simplex that we want to add to $\mathcal{C}(\partial C_{\dR + 1}(n))^{\leq d}$ does not contain all possible $\frac{\dR}{2}$-simplices, adding the $d$-simplex would also mean adding a missing $\frac{r}{2}$-face.
% Hence we get that the new complex would no longer be embeddable into $\R^{\dR}$.
But for general even $\dR$, not all $d$-faces of $C_{\dR+1}(n)$ have this form: $C_5(n)$, for example, has missing 3-faces.

\subsection{An Elementary Approach}

%The proof of Theorem \ref{NevoWagnerThm} uses rather involved topological methods. It traces the behavior of the van Kampen obstruction, a homological obstruction to embeddability, under bistellar operations which characterize piecewise-linear spheres. 
Is it possible to see directly why adding an additional simplex makes the complex $\mathcal{C}(\partial C_{2d + 1}(n))^{\leq d}$ no longer embeddable into $\R^{2d}$?

One possibility would be to find in the resulting complex a subcomplex of which we know that it is not embeddable.
 In Section \ref{embeddings} we saw examples of non-embeddable complexes: the joins of nice complexes (Theorem \ref{joins of nice complexes}).
They will be studied more deeply in an upcoming section. 
%We will now use results of an upcoming section in which these complexes will be studied more deeply.

Let us consider the case $d=2$.
In Section \ref{excludingJoins} we will see that there are only three possibilities for a join of nice complexes
where the dimension of the complex is $2$ and the dimension of the Euclidean space is $4$ (Lemma \ref{dimJoinNiceCompl}).
These complexes are 
\begin{itemize}
 \item $(\Delta_2^{\leq 0})^{*3}$ (the threefold join of the complex consisting of only three vertices, i.e., the join of the complete bipartite graph $K_{3,3}$ with the complex consisting of three vertices),
\item $\Delta_4^{\leq 1} * \Delta_2^{\leq 0}$ (the join of the complete graph $K_5$ with the complex consisting of three vertices) and
\item $\Delta_6^{\leq 2}$ (the $2$-skeleton of the $6$-simplex, i.e., all possible triangles on $7$ vertices).
\end{itemize}

\begin{figure}[ht]
\begin{center}
\includegraphics[scale=0.5]{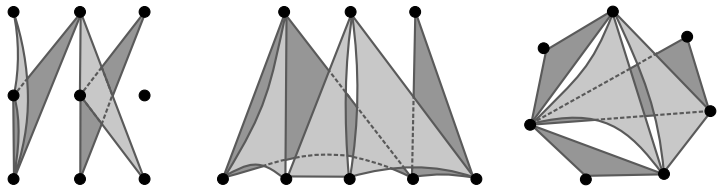}
\end{center}
\caption{Some exemplary $2$-simplices of the complexes $(\Delta_2^{\leq 0})^{*3}$, $\Delta_4^{\leq 1} * \Delta_2^{\leq 0}$ and $\Delta_6^{\leq 2}$.}
\end{figure}

We will now see that these do not have to appear when an additional \mbox{$2$-simplex} is added to the $2$-skeleton of a cyclic $5$-polytope. This shows that as subcomplexes the non-embeddable complexes from Theorem \ref{joins of nice complexes} do not yield a clearer picture of the non-embeddability of $\mathcal{C}(\partial C_{2d + 1}(n))^{\leq d} \cup \{M\}$ for a missing $d$-face $M$.

\begin{lem}\label{subcomplexes}
 There is a missing $2$-face $F$ of $C_5(9)$ such that the complex $\mathcal{C}(\partial C_5(9))^{\leq 2} \cup \{F\}$ contains none of the three complexes  $(\Delta_2^{\leq 0})^{*3}$, \mbox{$\Delta_4^{\leq 1} * \Delta_2^{\leq 0}$} and $\Delta_6^{\leq 2}$ as a subcomplex.%does not contain any
\end{lem}

\begin{proof}
 Let us first consider which sets $\{i,j,k\} \subset [9]$ with $i<j<k$ are not a vertex set of a $2$-face of $C_5(9)$. By Proposition \ref{faces} this is the case if and only if there is no facet of $C_5(9)$ that contains the vertices $i, j$ and $k$.

By Gale's Evenness Condition a set $S\subset[9]$ with $|S| = 5$ corresponds to a facet of $C_5(9)$ if and only if for all $j_1,j_2 \in ([9] \setminus S)$ with $j_1<j_2$ the number of $l \in S$ between $j_1$ and $j_2$ is even.
This means that all maximal intervals in $S$ that contain neither $1$ nor $9$ have even length, where ``maximal interval'' refers to a set $\{a,a+1,\ldots,a+k\}\subset S$ for which neither $\{a,a+1,\ldots,a+k+1\}$ nor $\{a-1,a,\ldots,a+k\}$ are subsets of $S$.

We will now check when a set $\{i,j,k\} \subset [9]$ with $i<j<k$ can be extended to such a $5$-set $S\subset[9]$.
Let us for the moment call an element $l$ of a set $S \subset [9]$ \emph{isolated} if $l\neq 1$, $l\neq 9$ and $l-1, l+1 \notin S$.
If there is no isolated vertex in $\{i,j,k\}$, it is possible to add two additional vertices such that the resulting set satisfies the above criterion for a facet.
Also if only one or two of $i$, $j$ and $k$ are isolated, we can add neighbors of the isolated vertex or vertices to create a facet.
These cases occur precisely if $i=1$ or $k=9$ or at least two of $i,j,k$ are adjacent.

If all three vertices $i$, $j$ and $k$ are isolated, we would need at least three additional vertices to make all maximal intervals not containing $1$ or $9$ have even length.
This shows that $\{i,j,k\}$ is not a $2$-face if and only if $i$, $j$ and $k$ are isolated, i.e., if $i \neq 1, k \neq 9$ and $i+1 < j < k-1$.

The complete list of missing 2-faces of $C_5(9)$ is therefore (with $\{i,j,k\}$ written as $ijk$):
\[
 246, 247, 248, 257, 258, 268, 357, 358, 368, 468
.\]
Remember that every set $\{i,j,k\}$ that does not correspond to a $2$-face of our polytope is a missing $2$-face as the neighborly polytope $C_5(9)$ has a complete $1$-skeleton.

\begin{figure}[ht]
\begin{center}
\psfrag{1}{1}
\psfrag{2}{2}
\psfrag{3}{3}
\psfrag{4}{4}
\psfrag{5}{5}
\psfrag{6}{6}
\psfrag{7}{7}
\psfrag{8}{8}
\psfrag{9}{9}
\includegraphics[scale=0.6]{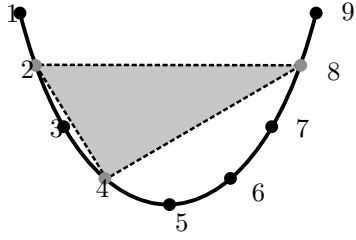}
\end{center}
\caption{A sketch of the vertices of $C_5(9)$ and the missing triangle $248$.}
\end{figure}

We will now show that $\mathcal{C}(\partial C_5(9))^{\leq 2} \cup \{248\}$ does not contain any of the three non-embeddable subcomplexes:
\begin{description}
\item[$(\Delta_2^{\leq 0})^{*3}:$]
 The nine vertices of this complex are partitioned into three sets of three vertices each, such that the complex contains precisely those $3$-sets that have one vertex in each of the three sets.

 For $248$ to be one of the simplices, each of the three sets would hence have to contain one of the vertices $2,4$ and $8$. 
 In which of the sets could we then put the vertex $6$? (We have to use all vertices!)
 Because $246, 468$ and $268$ are missing, $6$ cannot be contained in any of the three sets.
\item[$\Delta_4^{\leq 1} * \Delta_2^{\leq 0}:$]
 This complex has eight vertices, partitioned into one set $S_5$ of five
 and another set $S_3$ of three vertices.
 The simplices are exactly the sets that contain two vertices in $S_5$ and one in $S_3$.
 
 We thus have to have $|\{2,4,8\} \cap S_5| = 2$ and $|\{2,4,8\} \cap S_3| = 1$.

 The vertex $6$ cannot be a vertex of the subcomplex: If it was, it would have to form a simplex with two of the vertices $2,4$ and $8$.
 The missing faces $246, 468$ and $268$ prevent this.

 This means we have to use all other vertices: $1,2,3,4,5,7,8$ and $9$.
 We distinguish two cases:
\begin{description}
 \item[$2 \in S_5:$]
 Because $258$ is missing, $5$ has to be in the same partition set as $8$.
 The same is true for $4$ and $7$ ($247$). This means that $5$ and $7$ are not in the same set ($4$ and $8$ are not!).
 So $257$ would have to be a simplex!
 \item[$2 \in S_3:$]
 Then $4,8 \in S_5$ and, because of the missing faces $258$ and $247$, we would have $5,7 \in S_3$.
 This would mean that $S_3 = \{2,5,7\}$ and thus $3 \in S_5$. But then $358$ would have to be a simplex!
\end{description}
 \item[$\Delta_6^{\leq 2}:$]
 To find this complex as a subcomplex of $C_5(9)$, we have to choose four further vertices such that $248$ is the only missing $2$-face on these seven vertices. Because $246$ and $247$ are missing, we cannot pick $6$ or $7$.
 This leaves us with the seven vertices $1,2,3,4,5,8$ and $9$. But $358$ is also not a face of $C_5(9)$.
\end{description}
This shows that $\mathcal{C}(\partial C_5(9))^{\leq 2} \cup \{248\}$ does not contain any of the three complexes $(\Delta_2^{\leq 0})^{*3}$, $\Delta_4^{\leq 1} * \Delta_2^{\leq 0}$ and $\Delta_6^{\leq 2}$ as a subcomplex.
\end{proof}

So we see that, when trying to understand why $\mathcal{C}(\partial C_{2d + 1}(n))^{\leq d} \cup \{M\}$ for a missing $d$-face $M$ is not embeddable, it doesn't suffice to consider the non-embeddable complexes from Theorem \ref{joins of nice complexes} as subcomplexes.

We can, however, use a concept of minors for simplicial complexes to gain more insight into the non-embeddability of these complexes.
As already mentioned in Section \ref{minors}, Nevo introduces a notion of minors for finite simplicial complexes and explores their connection with embeddability in \cite{Nevo.2007}.

This is his definition of a minor:\label{defMinor}
\begin{Def}[minor]\label{definitionMinor}
Let $L$ and $K$ be simplicial complexes.
\begin{itemize}
\item If $L$ is a subcomplex of $K$, %$K \mapsto L$
the operation of replacing $K$ by $L$ is called a \emph{deletion}.
\item If $L$ is obtained from $K$ by identifying two distinct vertices that are not contained in any missing $k$-face
of $K$ with $k \leq \dim(K)$, %$K \mapsto L$
the operation of replacing $K$ by $L$ is called an \emph{admissible contraction}. See Figure \ref{admContr} for an example.
If $u,v \in V(K)$ are identified, the resulting complex is 
\[
 L = \left\{F \in K \;\middle|\; u \notin F\right\} \cup \left\{F  \cup \{v\} \;\middle|\; F \cup \{u\} \in K\right\}.
\]
\end{itemize}
We call $L$ a \emph{minor} of $K$ if $L$ can be obtained from $K$ by a (finite) sequence of admissible contractions and deletions.
We denote this situation by $L < K$.
\end{Def}

\begin{figure}[ht]
\begin{center}
\includegraphics[scale=0.45]{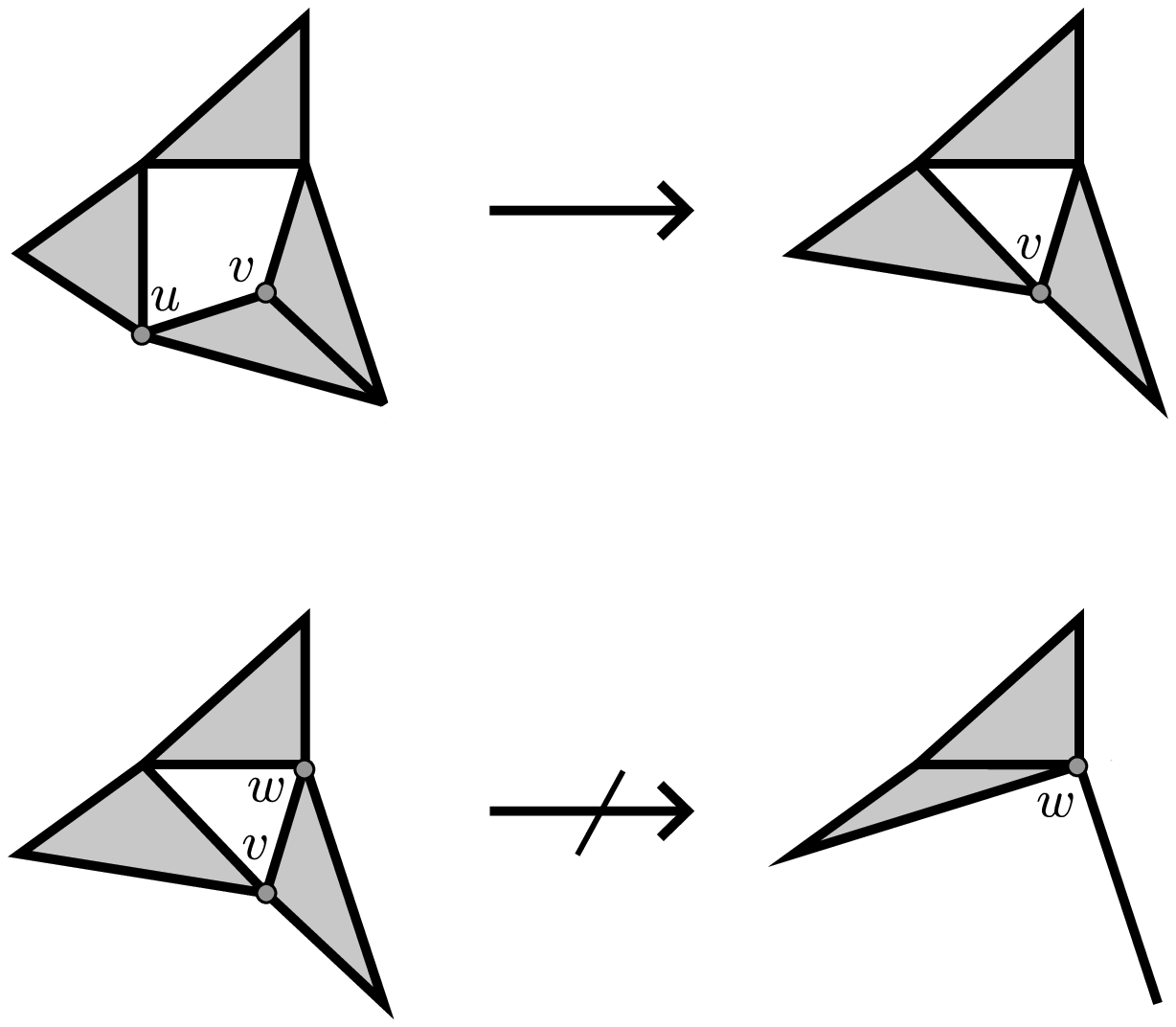}
\end{center}
\caption{An admissible and a non-admissible contraction.}\label{admContr}
\end{figure}

Nevo proves, among other things, the following statement about the $d$-skeleton of the $(2d+2)$-simplex which, by the Theorem of van Kampen and Flores (Theorem \ref{vK-Fl}), does not allow an embedding into $\R^{2d}$ and is also contained in the class of non-embeddable complexes of Theorem \ref{joins of nice complexes}.
\begin{thm}[{Nevo~\cite[Corollary 1.2]{Nevo.2007}}]
 For $d \geq 1$ let $K$ be a simplicial complex such that $(\Delta_{2d+2})^{\leq d}<K$ .
 Then $K$ cannot be embedded into $\R^{2d}$. 
\end{thm}
We now apply this to our question by showing that adding an additional simplex to $\mathcal{C}(\partial C_{2d + 1}(n))^{\leq d}$ creates $(\Delta_{2d+2})^{\leq d}$ as a minor and thus destroys embeddability into $\R^{2d}$.
\cite[Remark 6.4]{Nevo.2008} states this without proof.
Note that this proof of the non-embeddability of $\mathcal{C}(\partial C_{2d + 1}(n))^{\leq d} \cup \{M\}$ for a missing $d$-face $M$ is independent of Theorem \ref{NevoWagnerThm}.
\begin{lem}\label{zweiDeeMinor}
 Let $M$ be a missing $d$-face of $C_{2d + 1}(n)$ for $n \geq 2d+3$. 
 Then $(\Delta_{2d+2})^{\leq d}$ is a minor of $\mathcal{C}(\partial C_{2d + 1}(n))^{\leq d} \cup \{M\}$.
\end{lem}
\begin{proof}
 Let us first identify the missing $d$-faces of $C_{2d + 1}(n)$.
 By Gale's Evenness Condition and Proposition \ref{faces} we know that $S \subset [n]$ with $|S| = d+1$ is a $d$-face if and only if there are $d$ additional vertices such that these $2d+1$ vertices are arranged in one block of odd length containing either $1$ or $n$ and possibly several blocks of even length. Phrased with the terminology from the proof of Lemma \ref{subcomplexes} this means that the maximal intervals formed by these $2d+1$ vertices that do not contain $1$ or $n$ are of even length.

 It follows from this that $S$ is a $d$-face if and only if $1\in S$, $n \in S$ or it contains two adjacent vertices $i, i+1 \in S$, as these are the cases when there are $\leq d$ vertices that might need a neighbor. 
 Thus, for our missing face $M$ we have $M=\{i_1, \ldots , i_{d+1}\}$ with
 $i_1 \neq 1$, $i_{d+1}\neq n$ and $i_{k-1} + 1 < i_k < i_{k+1}-1$ for $2\leq k \leq d$. Again with the terminology of Lemma \ref{subcomplexes} this means that all elements of $M$ are isolated vertices.
 
\begin{figure}[ht]
\begin{center}
\psfrag{1}{$1$}
\psfrag{2}{$i_1$}
\psfrag{3}{$i_2$}
\psfrag{4}{$i_3$}
\psfrag{5}{$i_{d-2}$}
\psfrag{6}{$i_{d-1}$}
\psfrag{7}{$i_d$}
\psfrag{8}{$i_{d+1}$}
\psfrag{n}{$n$}
\psfrag{p}{$\ldots$}
\includegraphics[scale=0.6]{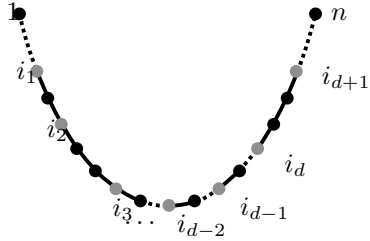}
\end{center}
\caption{A sketch of the vertices of $C_{2d + 1}(n)$ with the vertices of a missing face highlighted.}
\end{figure}

 Let $K = \mathcal{C}(\partial C_{2d + 1}(n))^{\leq d} \cup \{M\}$.
 By the above arguments, for $j \neq 1,n-1$ there is no missing $d$-face of $C_{2d + 1}(n)$ containing the vertices $j$ and $j+1$.
 Because $C_{2d + 1}(n)$ is $d$-neighborly, this means that identifying two vertices $j, j+1 \notin M$ with $j \neq 1,n-1$, if such vertices exist, is an admissible contraction in $K$.

 Let us now look at the resulting complex $K'$.
 Just like $K$ it contains all possible $k$-simplices for $k \leq d-1$.
 Which $d$-simplices are in $K'$?
 Certainly all $d$-simplices of $K$ not containing $j$ remain.
 Of the others, the ones that have $j$ and $j+1$ as elements turn into $(d-1)$-simplices.
 The ones that do not contain $j+1$ become $d$-simplices of the new complex, now containing $j+1$ instead of $j$.
 
 This means for $S \subseteq [n] \setminus {j}$ with $|S|=d+1$ we get
 \[
 S \in K' \Leftrightarrow 1 \in S\text{, } n \in S\text{, } (j-1,j+1 \in S) \text{ or } (i, i+1 \in S) \text{ for some } i,
 \]
 which proves that $K'$ is isomorphic to the $d$-skeleton of the cyclic $(2d+1)$-poly\-tope on $n-1$ vertices with $M$ as an additional simplex.
 
We can proceed to contract vertices in this manner until there is no more $j\notin M$ such that $j+1 \notin M$, $j \neq 1$ and $j+1 \neq n$.
This is the case if we are left with $2d+3$ vertices and our missing face $M$ contains every second vertex.

Thus, it suffices to consider the case $n=2d+3$ and $M = \{2,4,\ldots,2d+2\}$ where $M$ is the only missing $d$-face and hence $K$ the complete complex on $2d+3$ vertices: $K=(\Delta_{2d+2})^{\leq d}$.
% Since we can keep contracting vertices in this manner until we are left with $2d+3$ vertices, it suffices to consider the case $n=2d+3$ and $M = \{2,4,\ldots,2d+2\}$ where $M$ is the only missing $d$-face and hence $K$ the complete complex on $2d+3$ vertices: $K=(\Delta_{2d+2})^{\leq d}$.
\end{proof}

\noindent This concludes the part on lower bounds. We achieved a lower bound by looking at the cyclic polytopes as examples of embeddable complexes and could see that these examples cannot be improved by just adding more simplices.
\chapter{Upper Bounds}\label{upperBounds}
In this chapter, we now turn to upper bounds for Problem \ref{dieFrage}.
The idea we will pursue is to exclude certain non-embeddable complexes as subcomplexes.
In Section \ref{embeddings} we already introduced a class of complexes, %(Theorem \ref{joins of nice complexes}),
which seems to contain all known examples of minimal non-embeddable complexes.
We will recall their definition in Section \ref{excludingJoins}.
%The idea we will pursue is to exclude as subcomplexes the minimal non-embeddable complexes we introduced in Section \ref{embeddings} (Theorem \ref{joins of nice complexes}). %in previous sections.
A complex that has such a complex as a subcomplex is itself non-embeddable. Thus, if a certain number of simplices enforces these complexes as subcomplexes, it also enforces non-embeddability.

To this end, we will consider simplicial complexes as hypergraphs, a generalization of graphs which will be introduced at the start of this chapter. We will then use results from extremal hypergraph theory to estimate the number of edges that enforces as subgraphs the hypergraphs corresponding to our forbidden complexes.
We only get a result for certain instances:
We show that for large enough $n$, a $d$-dimensional simplicial complex on $n$ vertices that embeds in $\R^{2d}$ can have at most $n^{d+1-\frac{1}{3^d}}$ $d$-simplices, which slightly improves the trivial upper bound of $n^{d+1}$.

In the last part we will see that a better estimate for the number of edges enforcing these forbidden subgraphs could not improve this bound by much.

%for $\dR=2d$ we know: between n^d (last section) and n^{d+1}
%will see: less than n^{d+1}
\section{Hypergraphs}
%def
%k-uniform
%k-partit
%komplex <-> hypergraph
% n(G) = number of vertices, e(G)= number of edges
%def subhypergraph
We start by quickly introducing hypergraphs.
While edges of a graph always contain exactly two vertices, the edges of a hypergraph are arbitrary subsets of the set of vertices:

\begin{Def}[hypergraph]
 A \emph{(finite) hypergraph} is a pair $H=(V,E)$ where $V$ is a finite set and $E$ is a collection of subsets of $V$: $E\subseteq \mathcal{P}(E)$.
Elements of $V$ are called \emph{vertices}; the \emph{edges} of $H$ are the elements of $E$. We will denote the number of vertices of a hypergraph $H$ by $n(H)$ and its number of edges by $e(H)$.

A \emph{subhypergraph} of a hypergraph $H=(V,E)$ is a hypergraph $F=(V',E')$ with $V' \subseteq V$ and $E' \subseteq E$.
\end{Def}

Note that the edges of a hypergraph can have different cardinalities.
Hypergraphs where all edges have the same cardinality are called regular:

\begin{Def}[$d$-regular, $d$-graph]
 Let $d \geq 2$.
 A hypergraph $H$ is \emph{$d$-regular} or simply a \emph{$d$-graph} if all of its edges have cardinality $d$: $|S| = d$ for all $S \in E$.
\end{Def}

The objects of standard graph theory could, e.g., in this context be referred to as 2-graphs. We will continue calling them graphs where there is no ambiguity.

Comparing hypergraphs and abstract simplicial complexes ,which are also collections of subsets of some vertex set, one can observe one difference: Simplicial complexes are hereditary, i.e., for each edge they also contain all of its subsets. Hypergraphs don't have to be: a hypergraph could for example contain a ``triangle'' $\{a,b,c\}$ along with the edge $\{a,b\}$ but not the other two edges.
One could consider simplicial complexes as (non-regular) hereditary hypergraphs.
Pure $d$-dimensional simplicial complexes are in one-to-one correspondence with $(d+1)$-graphs: The set of all $d$-simplices of a pure $d$-dimensional complex $K$, denoted by $(K)^d$, is the edge set of a $(d+1)$-graph, and by adding all subsets of edges we can regain $K$ from it.

% Some notation and more definitions:
% \begin{enumerate}
%  \item We will denote the number of vertices of a hypergraph $H$ by $n(H)$ and its number of edges by $e(H)$.
% \item A \emph{subhypergraph} of a hypergraph $H=(V,E)$ is a hypergraph $F=(V',E')$ with $V' \subseteq V$ and $E' \subseteq E$.
% \end{enumerate}
%We will also need the following generalization of bipartiteness for 2-graphs:
We will mostly be interested in $d$-regular hypergraphs with an additional property, a generalization of bipartite (2-)graphs:

\begin{Def}[$d$-partite $d$-graph, complete $d$-partite $d$-graph]
 A $d$-graph $H=(V,E)$ is called \emph{$d$-partite} if its vertex set can be partitioned into $d$ pairwise disjoint sets $V=V_1 \cup V_2 \cup \ldots \cup V_d$, $V_i \cap V_j = \emptyset$ for $i \neq j$, such that $|S \cap V_i| \leq 1$ for every $i \in [d]$ and every $S \in E$.

The \emph{complete $d$-partite $d$-graph} $K^{(d)}_{l_1,l_2,\ldots,l_d}$ has a vertex set $V$ that can be partitioned into pairwise disjoint sets $V_1, V_2, \ldots, V_d$ with $|V_i|=l_i$ such that its edge set consists of all possible $d$-sets $S$ that contain exactly one element from each of the $V_i$: $E = \left\{S \subset V \;\middle|\; \text{ for every } i \in [d]\text{: }|S \cap V_i| = 1 \right\}$.
%  A $d$-regular hypergraph $H=(V,E)$ is called \emph{$k$-partite} for $k \geq d$ if its vertex set can be partitioned into $k$ pairwise disjoint sets $V_1, V_2, \ldots, V_k$, $V=V_1 \cup V_2 \cup \ldots \cup V_k$, $V_i \cap V_j = \emptyset$ for $i \neq j$, such that $|S \cap V_i| \leq 1$ for every $i \in [k]$ and every $S \in E$.
% 
% The \emph{complete $k$-partite $d$-graph} with vertex partition sizes $l_1, l_2,\ldots, l_k$ has a vertex set partitioned into pairwise disjoint sets $V_1, V_2, \ldots, V_k$
% with $|V_i|=l_i$. Its edge set consists of all possible $d$-sets $S$ for which $|S \cap V_i| \leq 1$ for every $i \in [k]$.
% This hypergraph is denoted by $K^{(d)}_{l_1,l_2,\ldots,l_k}$.
\end{Def}

\section{Some Extremal Hypergraph Theory}
We now turn to extremal problems on hypergraphs. The kind of question we are interested in is the following:
With fixed $k$ and $n$, what is the maximal number of edges a $k$-graph on $n$ vertices can have without containing certain subgraphs?
To address these questions we first introduce some terminology:

\begin{Def}[$\ex(n,\mathcal{F})$]
 Let $k \geq 2$ and let $\mathcal{F}$ be a family of $k$-graphs. A $k$-graph $H$ that contains no copy of any $F \in \mathcal{F}$ as a subgraph is called \emph{$\mathcal{F}$-free}.
By $\ex(n,\mathcal{F})$ we denote the maximal number of edges for an $\mathcal{F}$-free $k$-graph on $n$ vertices: \[
 \ex(n,\mathcal{F})=\max\left\{m\;\middle|\; H=(V,E) \text{ } \mathcal{F}\text{-free } k\text{-graph, } |V|=n,\, |E|=m\right\}.
\]
\end{Def}
While calculating $\ex(n,\mathcal{F})$ for a given $\mathcal{F}$ is in general very difficult, it is often possible to estimate $\ex(n,\mathcal{F})$ by lower and upper bounds.
We will now collect a few simple observations that are helpful to find estimates.

Observe that for every $F \in \mathcal{F}$ an $\mathcal{F}$-free graph is in particular $F$-free. This proves the following lemma:
\begin{lem}\label{exBoundedBySingleMember}
 For every family $\mathcal{F}$ of $k$-graphs  $\ex(n,\mathcal{F}) \leq \min_{F \in \mathcal{F}} \ex(n,F)$.
\end{lem}
Note that there are families $\mathcal{F}$ for which $\ex(n,\mathcal{F}) < \min_{F \in \mathcal{F}} \ex(n,F)$.
Consider, e.g., the two graphs $G$ and $H$ depicted in Figure \ref{exGr}.

It is easy to see that 
\[
 \ex(n,G) = \begin{cases}
             n-1 & n\neq3,\\
             3 & n=3
            \end{cases}
\]
 and $\ex(n,H) = \lfloor \frac{n}{2} \rfloor$, while $\ex(n,\{G,H\}) = 1$ for all $n \in \mathbb{N}$.

\begin{figure}[ht]
\begin{center}
\psfrag{G}{$G$}
\psfrag{H}{$H$}
\includegraphics[scale=0.7]{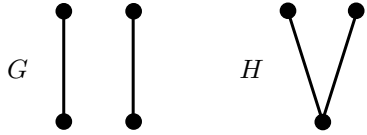}
\end{center}
\caption{A graph that contains neither of these graphs $G$ and $H$ can have at most one edge.}\label{exGr}
\end{figure}

For families of $k$-graphs not containing any $k$-partite hypergraphs, there is the following simple but strong lower bound:
\begin{lem}\label{partiteGraphs}
 If a family $\mathcal{F}$ of $k$-graphs has no $k$-partite member, there is a $c>0$ such that $\ex(n,\mathcal{F}) \geq c n^k$ for large enough $n$.
\end{lem}
\begin{proof}
Consider the complete $k$-partite $k$-graph $K^{(k)}_{l_1,l_2,\ldots,l_k}$ with $l_i=\lfloor\frac{n+i-1}{k}\rfloor$ which has $n$ vertices. It is called the \emph{complete equipartite $k$-graph} on $n$ vertices.
This $k$-graph is $\mathcal{F}$-free because all of its subgraphs are $k$-partite.
It has $\prod_{1 \leq i \leq k}\lfloor\frac{n+i-1}{k}\rfloor \geq (\frac{n-l}{k})^k$ edges for some $0 \leq l < k$.
\end{proof}
%A simple but strong bound for $k$-partite graphs:
%
% \begin{lem}\label{partiteGraphs}
%  If a family $\mathcal{F}$ of $k$-graphs has no $k$-partite member, we have $\ex(n,\mathcal{F}) \geq c n^k$ for some $c>0$.ODER????
% \end{lem}
% %
% \begin{proof}
% Consider the complete $k$-partite $k$-graph $K^{(k)}_{l_1,l_2,\ldots,l_k}$ with $l_i=\lfloor\frac{n+i-1}{k}\rfloor$, called the complete equipartite $k$-graph on $n$ vertices.
% It is $\mathcal{F}$-free because all of its subgraphs are $k$-partite. It has $\prod_{1 \leq i \leq k}\lfloor\frac{n+i-1}{k}\rfloor \geq (\frac{n-1}{k})^k$ edges.
% GENAUER????
% \end{proof}
%
As a $k$-graph on $n$ vertices can have at most as many edges as the complete $k$-graph, we have a trivial upper bound of $\ueber{n}{k} = \Theta(n^k)$. Note that the lower bound for non-$k$-partite $k$-graphs above is of the same order.

For complete $k$-partite $k$-graphs with partition sets of fixed size, Erd\H{o}s \cite{Erdos.1964} proved the following upper bound which we will later use:
\begin{thm}[{Erd\H{o}s~\cite[Theorem 1]{Erdos.1964}}]\label{Erdos}
Let $l, k\geq 2$.
Then there is $n_0 = n_0(k,l)$ such that $\ex(n,K^{(k)}_{l,\ldots,l}) < n^{k-\frac{1}{l^{k-1}}}$ for all $n > n_0$. 
\end{thm}
A survey on extremal hypergraph theory is \cite{Furedi.1991}, where this result can be found (without proof) in Chapter $4$.
%Kann man n_0(k,l) bestimmen??
\section{An Upper Bound for the Case \texorpdfstring{$\dR=2d$}{}}
%\section{Excluding Joins of Nice Complexes in the Case $\dR=2d$}
\label{excludingJoins}
%what are k-partite joins of nice complexes in this case?
%As mentioned before, we want to see how big a simplicial complex can get without containing any of the complexes from Theorem \ref{joins of nice complexes} as subcomplexes.
%Lemma \ref{partiteGraphs} tells us that only complexes which correspond to $d$-partite hypergraphs will yield an interesting bound.
We now use the above concepts to achieve an upper bound on the size %number of simplices 
of a simplicial complex not containing as subcomplexes any of the non-embeddable complexes we introduced in Section \ref{embeddings}, which is then also an upper bound for Problem \ref{dieFrage}. Recall Theorem \ref{joins of nice complexes}:

{%
\renewcommand\thethm{\savedtheoremnumber}%
\begin{thm}[{Schild~\cite[Theorem 3.1]{Schild.1993}}]
 Let $K_1, K_2,\ldots,K_s$ be simplicial complexes and suppose there are $n_1, n_2,\ldots,n_s \in \mathbb{N}$ such that $K_i$ is nice on $[n_i]$.
 Then $K = K_1*K_2*\ldots*K_s$ is not embeddable in $\R^{\dR}$ for ${\dR}=\left( \sum_{i=1}^s n_i \right)-s-2$.

 Except for the following two cases the complex $K$ is minimal non-embeddable,
 i.e., every proper subcomplex is linearly embeddable in $\R^{\dR}$:
 \begin{enumerate}
  \item All $K_i$ are simplices.
  \item $K=K_1$ is the boundary of a simplex with an additional vertex ($s=1$). 
 \end{enumerate}
\end{thm}
\addtocounter{thm}{-1}%
}%
Remember that a simplicial complex $K$ on at most $n$ vertices, with a vertex set identified with some $V\subseteq[n]$ is called $K$ \emph{nice} on $[n]$ if the following condition holds:
 \[
  F \subset [n] \Rightarrow F \in K \text{ or } [n] \setminus F \in K, \text{ but not both.}
 \]

We will consider the $(d+1)$-graphs corresponding to the complexes that appear in Theorem \ref{joins of nice complexes}. The last section showed that the $(d+1)$-partite ones among them are of particular interest.
For the case $\dR=2d$, where $d$ is the dimension of $K = K_1*K_2*\ldots*K_s$, we can show that these complexes have a certain form.
This enables us to characterize those among them that correspond to $(d+1)$-partite $(d+1)$-graphs. For $\dR \neq 2d$ it seems harder to find such a characterization.
This is why, from now on, we will only consider the case $\dR=2d$.

With the notation introduced in the last section we are thus interested in $\ex(n,\mathcal{F}_d)$ where $\mathcal{F}_d$ is the following family of hypergraphs:
\begin{Def}[$\mathcal{F}_d$]\label{EffDee}
 Let $d \geq 1$.
 Denote by $\mathcal{F}_d$ the family of hypergraphs induced by $d$-dimensional joins of nice complexes that by Theorem \ref{joins of nice complexes} do not embed in $\R^{2d}$:
\begin{equation*}
\begin{split}
 \mathcal{F}_d \mathrel{\mathop:}=\left\{(K_1*K_2*\ldots*K_s)^d \;\middle|\; K_i \text{ nice on }\right.& [n_i],\, \dim(K_1*K_2*\ldots*K_s)=d,\\
&\quad \left. n_1+n_2+\ldots+n_s = 2d +s+2 \right\}.
\end{split}
\end{equation*}
\end{Def}

Let us now try to find a different description of $\mathcal{F}_d$.
First, we observe that the dimension of a nice complex is bounded:

\begin{lem} \label{dimensionNiceComplex}
 If a simplicial complex $K$ with $\dim(K) = d$ is nice on $[n]$, we have $d+2 \leq n \leq 2d+3$.
\end{lem}
\begin{proof}
 Let $K$ be a simplicial complex with $\dim(K) = d$ that is nice on $[n]$.

 A $d$-dimensional complex does not contain any simplex $F$ with $|F| \geq d+2$.
 This means that $K$ has to contain all simplices $F$ with $|F| \leq n-d-2$.
 Hence, $d \geq n-d-3$ and therefore $n \leq 2d+3$.
 Furthermore, $K$ cannot contain $[n]$ because $\emptyset \in K$. Thus, $d \leq n-2$.
\end{proof}
With this lemma we can now determine the structure of those joins of nice complexes for which $\left( \sum_{i=1}^s n_i \right)-s-2 = 2d$ where $d$ is the dimension of the join:
\begin{lem}\label{dimJoinNiceCompl}
 %Let $K_1, K_2,\ldots,K_s$ be simplicial complexes and $n_1, n_2,\ldots,n_s \in \mathbb{N}$  such that $K_i$ is nice on $[n_i]$.
 For $1 \leq i \leq s$ let $K_i$ be a complex that is nice on $[n_i]$.
Set
\[
 d = \big( \sum_{i=1}^s d_i \big)+s-1 = \dim(K_1*K_2*\ldots*K_s),
\]
where $d_i=\dim(K_i)$, and suppose that
\[
 \big( \sum_{i=1}^s n_i \big)-s-2 = 2d.
\]

Then $K_i = (\Delta_{2d_i+2})^{\leq d_i}$ for $1 \leq i \leq s$, and therefore
\[
 \mathcal{F}_d = \big\{(K_1*K_2*\ldots*K_s)^d\;\big|\;K_i = (\Delta_{2d_i+2})^{\leq d_i},
 \sum_{i=1}^s d_i =d -s +1\big\}.
\]

\end{lem}
\begin{proof}
 The assumption
\[
 \left( n_1+n_2+\ldots+n_s \right)-s-2 = 2d = \left( 2d_1 + 2d_2 + \ldots 2d_s \right)+2s-2
\]
 simplifies to
\[
 \left( n_1+n_2+\ldots+n_s \right) = \left( 2d_1 + 2d_2 + \ldots 2d_s \right)+3s. 
\]
 Because we have $n_i \leq 2d_i +3$ by Lemma \ref{dimensionNiceComplex}, this shows that $n_i = 2d_i + 3$.

 This determines $K_i$:
 Let $F \subset [n_i]$ with $|F| \leq d_i+1$. Then $[n_i] \setminus F$ cannot be an element of $K_i$ because $|[n_i] \setminus F| \geq d_i+2$. Since $K_i$ is nice, it hence has to contain all such $F$.
 Therefore, $K_i = (\Delta_{2d_i+2})^{\leq d_i}$.
\end{proof}
This also shows that all complexes in $\mathcal{F}_d$ are minimal non-embeddable, since none of the two exceptional cases in Theorem \ref{joins of nice complexes} applies.

For $d=2$ this shows that in $\mathcal{F}_2$ there are only three complexes: $(\Delta_2^{\leq 0})^{*3}$, $\Delta_4^{\leq 1} * \Delta_2^{\leq 0}$ and $\Delta_6^{\leq 2}$. We already used this in Section \ref{improveLowerBound} where we showed that none of these three appears as a subcomplex when a missing face is added to the boundary of $C_5(9)$.

For $d=1$, the case of planar graphs, this yields that $\mathcal{F}_1$ consists exactly of the two graphs $K_5=\Delta_4^{\leq 1}$ and $K_{3,3}=(\Delta_2^{\leq 0})^{*2}$ that characterize the planarity of graphs.

By Lemma \ref{partiteGraphs} only the $(d+1)$-partite $(d+1)$-graphs in $\mathcal{F}_d$ can yield interesting upper bounds. The following lemma characterizes these:
\begin{lem}\label{partiteness}
 %Let $K=(\Delta_{2d_1+2})^{\leq d_1}*(\Delta_{2d_2+2})^{\leq d_2}* \ldots *(\Delta_{2d_r+2})^{\leq d_r}$
 %for some $d_1,d_2,\ldots,d_r \in \mathbb{N}$.
For $1\leq i\leq s$ let $d_i \in \mathbb{N}$ and $K_i = (\Delta_{2d_i+2})^{\leq d_i}$.
Furthermore, set $d = \dim(K_1*K_2*\ldots*K_s) = \left( \sum_{i=1}^s d_i \right)+s-1$.% d_1+d_2+\ldots+d_s+s-1$.

Then $K = K_1*K_2*\ldots*K_s$ corresponds to a $(d+1)$-partite hypergraph if and only if $d_i = 0$ for all $1 \leq i \leq s$. 
\end{lem}
\begin{proof}
 Suppose that at least one of the $K_i$ is not $0$-dimensional.
 We can assume w.l.o.g.\ that there is $1\leq t \leq s$ such that $d_i > 0$ for $i\leq t$ and $d_i = 0$ for $i>t$.

 To show $(d+1)$-partiteness we would have to partition the vertices of $K$ into $d+1$ subsets
 such that for any two vertices in a common subset there is no simplex containing both vertices.
 For such a partition $V_1 \cup V_2 \cup \ldots \cup V_{d+1} = V(K)$ we have:
 \begin{itemize}
  \item Any two vertices in different $K_i$ constitute a simplex of the join and hence have to be contained in different $V_j$:
       \[v_1 \in V(K_{i_1}), v_2 \in V(K_{i_2}) \text{ with } i_1 \neq i_2 \Rightarrow v_1 \in V_{j_1}, v_2 \in V{j_2} \text{ and } j_1 \neq j_2.\]
  \item For $i \leq t$ the complex $K_i$ has a complete $1$-skeleton, therefore its vertices each need to be contained in different $V_j$:
       \[v, w \in V(K_i) \text{ for } i \leq t \Rightarrow v \in V_{j_1}, w \in V_{j_2} \text{ and } j_1 \neq j_2.\]
 \end{itemize}
Thus, the number of parts of the partition would at least have to be
%
% \begin{equation*}
%  \begin{split}
%   n_1 + n_2 + \ldots + n_t + (s-t) & =  2(d_1 + d_2 + \ldots + d_t) +3t+ (s-t)\\
% 				   & >  d_1+d_2+\ldots+d_t+s = d+1.
%  \end{split}
% \end{equation*}
%
\[
 (s-t) + \sum_{i=1}^t|V(K_i)|  = (s-t) +3t+ \sum_{i=1}^t 2d_i > s+\sum_{i=1}^t d_i = d+1.
\]
This shows that, if such a $t$ exists, $K$ cannot be $(d+1)$-partite.
Hence, all of the $K_i$ have to be $0$-dimensional for $K$ to be $(d+1)$-partite.

Now assume that $d_i=0$ for all $i$. Then $K=((\Delta_2)^{\leq 0})^{*s}$, the $s$-fold join of the complex consisting of only three vertices. This complex has dimension $d=s-1$. Because no simplex of $K$ can contain two vertices from the same copy of $(\Delta_2)^{\leq 0}$, its vertex set can be partitioned into $s=d+1$ subsets each containing the three vertices of one $K_i=(\Delta_2)^{\leq 0}$.
\end{proof}
Thus, there is only one hypergraph we can exclude in search for an upper bound: the one corresponding to $((\Delta_2)^{\leq 0})^{*(d+1)}$.
We just proved that this hypergraph is $(d+1)$-partite, where the partition of the vertex set consists of $d+1$ sets of size $3$.
Given this partition, every set containing exactly one vertex from each partition set is an edge.
This shows that this hypergraph is the complete $(d+1)$-partite $(d+1)$-graph $K^{(d+1)}_{3,3,\ldots,3}$. (See Figure \ref{k333}.)

\begin{figure}[htbp]
\begin{center}
\includegraphics[scale=0.6]{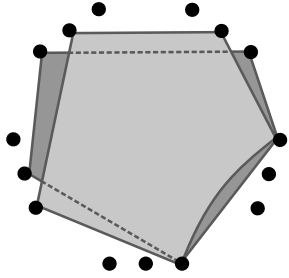}
\end{center}
\caption{Two exemplary edges of the $5$-graph $K^{(5)}_{3,3,3,3,3}$}\label{k333}
\end{figure}

% Hence, Theorem \ref{Erdos} gives us the desired upper bound:
% \[
% \ex(n,K^{(d+1)}_{3,3,\ldots,3}) < n^{d+1-\frac{1}{3^d}} \text{ for } n \geq n_0(3,d)
% .\]
% So we have:
Hence, Theorem \ref{Erdos} yields the desired upper bound for Problem \ref{dieFrage}:
\begin{prop}\label{theUpperBound}
 Let $d \geq 1$ and $n \geq 2d+1$.
 Then 
 \[
\max\left\{f_d(K)\;\middle|\; \dim(K)=d, |V(K)|=n, K \hookrightarrow \R^{2d}\right\} = O(n^{d+1-\frac{1}{3^d}})
.\]
\end{prop}
\begin{proof}
A $d$-dimensional simplicial complex that embeds in $\R^{2d}$ has to correspond to an $\mathcal{F}_d$-free hypergraph, as it can not have a subgraph that is not embeddable into $\R^{2d}$.
Thus, for $n \geq n_0(3,d)$ we have
\begin{align*}
 \max\left\{f_d(K)\;\middle|\; \dim(K)=d, |V(K)|=n, K \hookrightarrow \R^{2d}\right\} & \leq \ex(n,\mathcal{F}_d)\\
& \leq \ex(n,K^{(d+1)}_{3,3,\ldots,3})\\
& < n^{d+1-\frac{1}{3^d}},
\end{align*}
where the second inequality holds by Lemma \ref{exBoundedBySingleMember} and Lemma \ref{partiteness}, the third by Theorem \ref{Erdos}.
\end{proof}
This bound slightly improves the trivial upper bound of $\ueber{n}{d+1}=\Theta(n^{d+1})$, but does not get near the conjectured bound in Conjecture \ref{Conj3} and the lower bound of $\Omega(n^d)$ that we achieved in Chapter \ref{lowerBounds}.

Our estimate for $\max\left\{f_d(K)\;\middle|\; \dim(K)=d, |V(K)|=n, K \hookrightarrow \R^{2d}\right\}$ is hence not very precise.
In the case $d=2$ we now have:
\[
 2(n^2-6n+10) \leq \max\left\{f_2(K)\;\middle|\; \dim(K)=2, |V(K)|=n, K \hookrightarrow \R^4\right\} < n^{3-\frac{1}{9}}
\]
for large enough $n$.

Observe that for graphs ($d=1$) we get an upper bound of the order $O(n^{\frac{5}{3}})$.
One might think that $\ex(n,\mathcal{F}_1) = \ex(n,\{K_5,K_{3,3}\})$ is linear in $n$. But remember that we are only excluding $K_5$ and $K_{3,3}$ as subgraphs, not as subdivisions or minors! Actually, we will see in the upcoming sections that $\ex(n,\mathcal{F}_1) \geq n^{\frac{3}{2}}$.

\section{Improving the Upper Bound?}
%Just as in Section \ref{improveLowerBound} we now want to consider whether the upper bound we just gained can be improved.
In the previous section, we bounded the number of maximal simplices in a \mbox{$d$-complex} that embeds into $\R^{2d}$ by giving an upper bound for $\ex(n,\mathcal{F}_d)$. 
We now want to study how much could be gained by estimating $\ex(n,\mathcal{F}_d)$ more precisely.

To achieve this, we will, by means of the probabilistic method, prove a lower bound for $\ex(n,\mathcal{F})$ for any finite family $\mathcal{F}$ of $k$-graphs in which each member has at least $2$ edges.
This will show that for each $d$ and for large enough $n$, there has to be an $\mathcal{F}_d$-free $d$-graph with $n$ vertices and at least $C \cdot n^{d+1-\frac{2(d+1)}{3^{d+1}-1}}$ edges for some $C > 0$.
We first introduce the concept of random hypergraphs and state a few basic results which we will need for the proof.
%description of probabilistic method??

\subsection{Random Hypergraphs}
Introductions to random (2-)graphs can be found in \cite[Chapter 10]{Alon.2000} and \cite[Chapter 5]{Mitzenmacher.2005}.
A concept for random hypergraphs is mentioned in \cite{Janson.2000}.
%Random graphs bollobas?
We assume knowledge of basic probability theory.

We first describe the Erd\H{o}s-R\'enyi model for random hypergraphs, as this is the model that we will be using.
The idea is to build a random $k$-graph on the vertex set $[n]$ edge by edge, including each edge $e \subset \ueber{[n]}{k}$ with probability $p$.
% by deciding for each edge $e \subset \ueber{[n]}{k}$ independently whether it should belong to $G$, 
%where $e$ will be an edge of $G$ with probability $p$, a non-edge with probability $1-p$.
Formally, this works as follows:
\begin{Def}[$G^{(k)}(n,p)$]
 Let $n,k \in \mathbb{N}$, $n \geq k \geq 2$ and $p \in [0,1]$.

 For a set $e \subset \ueber{[n]}{k}$, we consider the Bernoulli trial with sample space 
\[
 \Omega_e=\{ e \in E(G), e \notin E(G)\}
\]
 and measure $P_e:\Omega_e \to [0,1]$ where 
\[
 P_e(e \in E(G)) = p \text{ and } P_e(e \notin E(G))=1-p.
\]

 Running this trial independently for every set $e \subset \ueber{[n]}{k}$ yields the probability space ($\prod_{e \subset \ueber{[n]}{k}} \Omega_e, \prod_{e \subset \ueber{[n]}{k}} P_e)$, which we denote by $G^{(k)}(n,p)$.

Its sample space can be identified with the set of all possible $k$-graphs on the vertex set $[n]$. The underlying measure is
\[
 P(G) = p^m (1-p)^{\ueber{n}{k}-m} \text{ for a } k \text{-graph } G \text{ with } |E(G)|=m.
\]
We will write $G \in G^{(k)}(n,p)$ to denote that we consider the $k$-graph $G$ as an elementary event in $G^{(k)}(n,p)$.
\end{Def}
The number of ``successes'' in this series of Bernoulli trials is the number of edges of the resulting hypergraph. We thus get:

\begin{lem} \label{numberEdges}
 Let $n \geq k \geq 2$ and $p \in [0,1]$. Let $E(G)$ denote the number of edges of $G \in G^{(k)}(n,p)$.
 Then $E$ is a binomially distributed random variable with parameters $\ueber{n}{k}$ and $p$
 and hence $\mathbb{E}(E) = \ueber{n}{k} p$.
\end{lem}
%
% \begin{proof}
%  Choosing or not choosing an edge is a Bernoulli trial.
% \end{proof}
%
We will also need an estimate for the expected number of subgraphs of a random $k$-graph that are isomorphic to a certain fixed hypergraph.
\begin{lem} \label{numberSubgraphs}
 Let $n \geq k \geq 2$, $p \in [0,1]$ and $H$ a $k$-graph with $n_H \leq n$ vertices and $e_H$ edges.
 Let $X_H(G)$ denote the number of subgraphs of $G \in G^{(k)}(n,p)$ that are isomorphic to $H$.

 Then $\mathbb{E}(X_H) \leq C_H \cdot n^{n_H} \cdot p^{e_H}$ 
where $C_H = \left|\left\{H' \subseteq \ueber{[n_H]}{k}\;\middle|\; H' \cong H\right\}\right| \cdot \frac{1}{n_H!}$.
%for some $C_H > 0$ depending on $H$.
\end{lem}
\begin{proof}
 For $M \subseteq \ueber{[n]}{k}$ let $\chi_M\!:\!G^{(k)}(n,p) \rightarrow \{0,1\}$ be the indicator function of $M$,
 i.e., for $G \in G^{(k)}(n,p)$ we have
\[
\chi_M(G)= \begin{cases} 
              1 & M \subseteq G; \\
              0 & \text{otherwise.}
	         \end{cases}
\]
Then
\[
 \mathbb{E}(\chi_{M}) = 0 \cdot P(M \nsubseteq G) + 1 \cdot P(M \subseteq G) = p^{|M|}.
\]
With this notation we can express $X_H(G)$ as the sum over $\chi_{H'}(G)$ for all ${H' \subseteq \ueber{[n]}{k}}$ that are isomorphic to $H$.
We thus get:
\[
 \mathbb{E}(X_H) = 
 \mathbb{E}\left(\sum_{H'} \chi_{H'}\right) =
 \sum_{H'} \mathbb{E}(\chi_{H'})
= \sum_{H'} p^{e_H},
\]
where the sums range over all $H' \subseteq \ueber{[n]}{k}$ that are isomorphic to $H$.
%Because $X_H(G) = \sum_{H' \subseteq \ueber{[n]}{k},H' \cong H} \chi_{H'}(G)$ we get
% \[
%  \mathbb{E}(X_H) = 
%  \mathbb{E}(\sum_{\substack{H' \subseteq \ueber{[n]}{k}\\H' \cong H}} \chi_{H'}) =
%  \sum_{\substack{H' \subseteq \ueber{[n]}{k},\\H'\cong H}} \mathbb{E}(\chi_{H'})
% = \sum_{\substack{H' \subseteq \ueber{[n]}{k},\\H' \cong H}} p^{e_H}.
% \]
The number of such $H'$ is $c_H \ueber{n}{n_H}$
where $c_H$ is the number of subsets of $\ueber{[n_H]}{k}$ that are isomorphic to $H$.
%Thus there is $C_H > 0$ such that
Thus for $C_H = c_H \cdot \frac{1}{n_H!}$ we have
\[
 \mathbb{E}(X_H) = c_H \cdot \ueber{n}{n_H} \cdot p^{e_H} \leq C_H \cdot n^{n_H} \cdot p^{e_H}.
\]
\end{proof}
We will furthermore need two well known inequalities, Markov's inequality and Chernoff's inequality, to estimate the probability of events describing the difference of a random value and its expected value.
%with which a random variable exceeds its expected value.
\begin{thm}[{Markov's inequality, e.g., \cite[Theorem 3.1]{Mitzenmacher.2005}}] \label{Markov}
 Let $X$ be a random variable with only non-negative values and let $a > 0$.
 Then
\[
P(X > a \cdot \mathbb{E}(X)) < 1/a. 
\]
\end{thm}
%
%proof?
%
\begin{thm}[{Chernoff's inequality, e.g., \cite[Corollary 4.6]{Mitzenmacher.2005}}] \label{Chernoff}
 Let $X$ be a binomially distributed random variable and $0 < b < 1$.
 Then
\[
 P(|X-\mathbb{E}(X)| \geq b \cdot \mathbb{E}(X)) \leq 2 e^{-b^2 \mathbb{E}(X)/3}.
\]
\end{thm}
The following lemma, with which we conclude this section, is a statement on general probability spaces, the generalization of the simple rule
\[
 P(A), P(B) > \frac{1}{2} \implies P(A\cap B)>0.
\]
\begin{lem} \label{intersectionEvents}
 Let $\Omega$ be a probability space and $A_1, \ldots, A_n \subseteq \Omega$ events with
 $P(A_i) > \frac{n-1}{n}$ for all $1\leq i \leq n$.
 Then the intersection of the $A_i$ has positive probability: $P(\bigcap_{i=1}^{n}A_i)>0$.
\end{lem}
\begin{proof}
 For an event $A$ let $\overline{A}$ denote the complementary event $\Omega\setminus A$.
 We have $P(\overline{A_i}) < \frac{1}{n}$ and therefore $P(\bigcup_{i=1}^{n}\overline{A_i}) \leq \sum_{i=1}^{n} P(\overline{A_i}) < 1.$
 Hence,
\[
 P(\bigcap_{i=1}^{n}A_i) = 1 - P(\overline{\bigcap_{i=1}^{n}A_i}) = 1 - P(\bigcup_{i=1}^{n}\overline{A_i}) > 0.
\]

\end{proof}

\subsection{A Lower Bound on \texorpdfstring{$\ex(n, \mathcal{F}_d)$}{}}

Let $\mathcal{F}$ be a finite family of $k$-graphs in which each member has at least $2$ edges.
We will now use the probabilistic tools we just introduced to prove a lower bound for $\ex(n,\mathcal{F})$ for large enough $n$.
We will show that if every hypergraph in $\mathcal{F}$ has less than $n_0$ vertices and more than $m_0>1$ edges, there is a $C>0$ such that for large $n$ there has to be an $\mathcal{F}$-free $k$-graph with $n$ vertices and at least $C \cdot n^{k-\frac{n_0 - k}{m_0 - 1}}$ edges.

At the end of this section we will then apply this to the family $\mathcal{F}_d$, achieving 
\[
\ex(n,\mathcal{F}_d) \geq C \cdot n^{d+1-\frac{2(d+1)}{3^{d+1}-1}} 
\]
for some $C>0$ and for large enough $n$.

 The idea for the proof of the general lower bound is similar to the one for proving the existence of graphs with large girth and large chromatic number (e.g., \cite[p. 38]{Alon.2000}):
We will show that there is a hypergraph $G$ which only contains ``few'' of the forbidden subgraphs in $\mathcal{F}$ while having ``a lot'' of edges, by proving that the probability of the event ``$G \in  G^{(k)}(n,p)$ has many edges and only few forbidden subgraphs'' is greater than zero.
We can choose the numbers of edges and subgraphs such that we can then delete one edge from every forbidden subgraph in $G$ and still keep sufficiently many edges.
 
\begin{prop}[\cite{Schacht}]\label{lowerBoundEx(n,F)}
 For $k \geq 2$ let $\mathcal{F}$ be a finite family of $k$-uniform hypergraphs and $n_0 \geq k$, $m_0 > 1$ with $n_0 < k \cdot m_0$ such that
\[
 n(F) \leq n_0 \text{ and } e(F) \geq m_0 \text{ for all } F \in \mathcal{F}.
\]
 Then there is $C = C(k) >0$ and $N \in \mathbb{N}$ such that
for every $n \geq N$ there exists an $\mathcal{F}$-free $k$-graph on $n$ vertices
 that has at least $C \cdot n^{k-\frac{n_0 - k}{m_0 - 1}}$ edges, i.e., 
\[
 \ex(n,\mathcal{F}) = \Omega(n^{k-\frac{n_0 - k}{m_0 - 1}}).
\]
\end{prop}
\begin{proof}
 Let $p=c_1n^{-c_2}$ for some $c_1,c_2$ such that $0 \leq p \leq 1$ for all $n \in \mathbb{N}$.
 For $G \in G^{(k)}(n,p)$ and $F \in \mathcal{F}$ let $X_F(G)$ denote the number of subgraphs of $G$ that are isomorphic to $F$; $E(G)$ shall denote the number of edges of $G$.

 We want to show that we can choose $c_1$ and $c_2$ such that a suitable event of the form
 ``$E(G)$ is large and $X_F(G)$ is small for all $F \in \mathcal{F}$'' has positive probability.
 Lemma \ref{intersectionEvents} shows that such an intersection of events has probability greater than zero if each of the events occurs with a certain probability.
 Markov's Inequality (Theorem \ref{Markov}) and Chernoff's Inequality (Theorem \ref{Chernoff}) are the tools we will use to estimate the probabilities of the single events.
 
 Let $l= |\mathcal{F}| + 1$.
 The events we want to consider are $A_F \mathrel{\mathop:}= \left\{X_F \leq l \cdot \mathbb{E}(X_F)\right\}$ for $F \in \mathcal{F}$
 and $A \mathrel{\mathop:}= \left\{|E - \mathbb{E}(E)| < \frac{1}{2} \cdot \mathbb{E}(E)\right\} = \left\{\frac{1}{2} \cdot \mathbb{E}(E) < E < \frac{3}{2} \cdot \mathbb{E}(E)\right\}$.

 If the intersection of these events occurs, i.e., $G \in \bigcap_{F \in \mathcal{F}} A_F \cap A$, then $G$ contains at most $l \cdot \mathbb{E}(X_F)$ copies of each 
$F \in \mathcal{F}$ and has at least $\frac{1}{2} \cdot \mathbb{E}(E)$ edges.%a bounded number of edges:  $\frac{1}{2} \cdot \mathbb{E}(E) < E(G) < \frac{3}{2} \cdot \mathbb{E}(E)$.

 Do these events have large enough probability?
 For $A_F$, Markov's Inequality shows that
\[
 P(A_F) = 1 - P(X_F > l \cdot \mathbb{E}(X_F)) > \frac{l-1}{l} \text{ for every } F \in \mathcal{F}.
\]
To estimate the probability of the event $A$ we first consider its complementary event.
 Because $E$ is a binomially distributed random variable with parameters $\ueber{n}{k}$ and $p$ (Lemma \ref{numberEdges}), Chernoff's Inequality yields
\begin{align*}
 P\left(|E - \mathbb{E}(E)| \geq \frac{1}{2} \cdot \mathbb{E}(E)\right) &\leq 2 e^{-\frac{1}{4} \mathbb{E}(E)/3}\\
               						     &= 2 e^{-\frac{1}{12} \ueber{n}{k} p}\\
           			              	             &\leq  2 e^{-\frac{1}{12} \frac{1}{2k!} n^k c_1 \cdot n^{-c_2}}
                                                                \text{ for } n \text{ large enough}\\
						 	     &< \frac{1}{l}  \text{ for } c_2 < k \text{ and } n \text{ large enough}.
\end{align*}
 Hence, if $c_2 < k$ and $n$ is large enough, we get 
\[
 P(A) = 1 -  P\left(|E - \mathbb{E}(E)| \geq \frac{1}{2} \cdot \mathbb{E}(E)\right) > \frac{l-1}{l}.
\]
 Lemma \ref{intersectionEvents} now tells us that
\[
 P\left(\bigcap_{F \in \mathcal{F}} A_F \cap A\right) > 0.
\]
 Thus, for $c_2 < k$ and large enough $n$ there is a $k$-graph $G$ on $n$ vertices that contains at most $l \cdot \mathbb{E}(X_F)$ copies of each $F \in \mathcal{F}$ and has at least $\frac{1}{2}\mathbb{E}(E)$ edges.

 Now, we want to choose $c_1$ and $c_2$ such that after removing one edge from every forbidden subgraph $F \in \mathcal{F}$ sufficiently many edges remain.
%there still remain a lot of edges.
 The number of edges we have to remove is at most $\sum_{F \in \mathcal{F}} l \cdot \mathbb{E}(X_F)$.
 So, if we want to keep at least half of the edges we have to have
\begin{equation}\label{gleichung}
 \sum_{F \in \mathcal{F}} l \cdot \mathbb{E}(X_F) < \frac{1}{4} \mathbb{E}(E).
\end{equation}
 How large is $\sum_{F \in \mathcal{F}} l \cdot \mathbb{E}(X_F)$?
 Lemma \ref{numberSubgraphs} tells us that for $F \in \mathcal{F}$ there exists $C_F \geq \frac{1}{n(F)!}$ such that
\[
 \mathbb{E}(X_F) \leq C_F \cdot n^{n(F)} \cdot p^{e(F)}.
\]
Therefore, setting $\bar{C}=l \cdot (\max_{F \in \mathcal{F}}C_F) \cdot |\mathcal{F}| = l (l-1) (\max_{F \in \mathcal{F}}C_F) >0$, we have
\[
 \sum_{F \in \mathcal{F}} l \cdot \mathbb{E}(X_F) \leq  l \cdot \sum_{F \in \mathcal{F}} C_F \cdot n^{n(F)} \cdot p^{e(F)} \leq  \bar{C} \cdot n^{n_0} \cdot p^{m_0}.
\]
What about the other side of \eqref{gleichung}?
With $\tilde{C} = \frac{1}{8k!} > 0$ we get
\[
  \frac{1}{4} \mathbb{E}(E) = \frac{1}{4} \ueber{n}{k} p \geq \frac{1}{4} \frac{1}{2k!} n^k p = \tilde{C} n^k p.
\]
Hence, \eqref{gleichung} is fulfilled if $\bar{C} \cdot n^{n_0} \cdot p^{m_0} < \tilde{C} n^k p$, which is equivalent to
\[
  C^* < n^{(k-n_0)} p^{(1-m_0)} \text{ where } C^* = \frac{\bar{C}}{\tilde{C}} = 8k! l (l-1) (\max_{F \in \mathcal{F}}C_F) > 1.
\]
%where $  C^* = \frac{\bar{C}}{\tilde{C}} = 8k! l (l-1) (\max_{F \in \mathcal{F}}C_F)$.
Because $p=c_1n^{-c_2}$,
\[
 n^{(k-n_0)} p^{(1-m_0)} = n^{(k-n_0-c_2(1-m_0))} c_1^{(1-m_0)} = n^{(c_2(m_0-1)-(n_0-k))} c_1^{-(m_0-1)}.
\]
This means, we want to find values for $c_1$ and $c_2$ such that
\begin{equation}\label{gleichung2}
 c_1^{(m_0-1)} C^* < n^{(c_2(m_0-1)-(n_0-k))}.
\end{equation}
If we choose
\[
 c_2 = \frac{(n_0-k)}{(m_0-1)} \text{ and } c_1 < \left(\frac{1}{C^*}\right)^{\frac{1}{m_0-1}},
\]
we have $n^{(k-n_0-c_2(1-m_0))} = 1$ and the left side of \eqref{gleichung2} is $<1$.
Thus, \eqref{gleichung2}, and therefore also \eqref{gleichung}, is fulfilled.

Is this an admissible choice of these parameters?
Two things have to be considered: We need $0<p=c_1n^{-c_2}<1$ and, to guarantee that $P(A)$ is large enough, $c_2 < k$.

If we choose $c_1 > 0$ we have $p>0$.
Since 
\[
 c_2 = \frac{(n_0-k)}{(m_0-1)} > 0 \text{ and }
 c_1 < \left(\frac{1}{C^*}\right)^{\frac{1}{m_0-1}} < 1
\]
we also get that $p<1$.
The last condition is also fulfilled because by assumption $n_0 < k \cdot m_0$ which is equivalent to
\[
 \frac{(n_0-k)}{(m_0-1)} < k.
\]

Thus, for $c_2 = \frac{n_0-k}{m_0-1}$ and suitable $c_1$ after removing at most half of the edges of $G$ we get an $\mathcal{F}$-free graph $G'$ that has at least
$\frac{1}{4} \mathbb{E}(E) > \tilde{C} n^k c_1 n^{-\frac{n_0-k}{m_0-1}} = C n^{k-\frac{n_0-k}{m_0-1}}$ edges for $C= \tilde{C} \cdot c_1 >0$.
\end{proof}
\begin{rem}
 Proving $\ex(n,\mathcal{F}) \geq C \cdot n^{k-a}$ for some $\frac{n_0 - k}{m_0 - 1} < a < k$ is a little bit easier.
 One can simply set $p=n^{-a}$ and, by proceeding as above, get to ${C^* < n^{(a(m_0-1)-(n_0-k))}}$ which is fulfilled for $a < \frac{n_0 - k}{m_0 - 1}$ and large enough $n$.
\end{rem}

To apply this result to the family $\mathcal{F}_d$, we need to know an upper bound for the number of vertices and a lower bound for the number of edges of the hypergraphs in $\mathcal{F}_d$.
The following lemma shows that the hypergraph $K^{(d+1)}_{3,3,\ldots,3}$, which corresponds to the complex $((\Delta_2)^{\leq 0})^{*(d+1)}$, determines both of these bounds.

\begin{lem}\label{maxVerticesminEdges}
 Let $d\geq 1$.
 %Among all $d$-dimensional joins of nice complexes that do not embed in $\R^{2d}$,%das stimmt so nicht!!
% i.e.\ complexes $K$ of the form $K = K_1*K_2*\ldots*K_r, K_i = (\Delta_{2d_i+2})^{\leq d_i}, d = d_1+d_2+\ldots+d_r+r-1$,
Among all $d$-complexes corresponding to hypergraphs in $\mathcal{F}_d$ the complex $((\Delta_2)^{\leq 0})^{*(d+1)}$ has the maximal number of vertices and the minimal number of $d$-simplices.
\end{lem}
\begin{proof}
 Let $K$ be a simplicial complex that corresponds to an element of $\mathcal{F}_d$.
By Lemma \ref{dimJoinNiceCompl} there are $r$ and $d_1,d_2,\ldots,d_r$ with $d = r - 1 + \sum_{i=1}^r d_i$ such that $K = K_1*K_2*\ldots*K_r$ where $K_i = (\Delta_{2d_i+2})^{\leq d_i}$.

 We have $V(K) = \sum_{i=1}^{r}2d_i+3 = 2(d_1+d_2+\ldots+d_r) +3r = 2d +r + 2$.
 Hence, the number of vertices of $K$ is maximal if $r$ is maximal.
 But $r= d+1-(d_1+d_2+\ldots+d_r)$. So, $r$ is maximal if $d_i=0$ for all $1 \leq i \leq r$.
 The number of $d$-simplices of $K$ is $f_d(K) = \prod_{i=1}^{r}f_{d_i}(K_i) = \prod_{i=1}^{r}\ueber{2d_i+3}{d_i+1}$.
 The sequence $\ueber{2n+3}{n+1}$ is strictly increasing because 
\begin{equation*}
 \begin{split}
 \ueber{2(n+1)+3}{(n+1)+1} & = \ueber{2n+5}{n+2}\\
                           & = \ueber{2n+4}{n+1} + \ueber{2n+4}{n+2}\\
                           & = \ueber{2n+3}{n}+\ueber{2n+3}{n+1}+\ueber{2n+3}{n+1}+\ueber{2n+3}{n+2}\\
                           & > \ueber{2n+3}{n+1}.
\end{split}
\end{equation*}
Hence, each of the factors is minimal for $d_i=0$.
\end{proof}

With this we get the desired lower bound for $\ex(n,\mathcal{F}_d)$:

\begin{cor}[\cite{Schacht}]\label{lowerBoundExEffDee}
 Let $d\geq 1$, and $\mathcal{F}_d$ as defined in Definition \ref{EffDee}.
 Then there is a $C>0$ and $N \in \mathbb{N}$ such that
\[
 \ex(n,\mathcal{F}_d) \geq C \cdot n^{d+1-\frac{2(d+1)}{3^{d+1}-1}} \text{ for } n \geq N.
\]
\end{cor}
\begin{proof}
 Lemma \ref{maxVerticesminEdges} tells us that $n(F) \leq n(K^{(d+1)}_{3,3,\ldots,3})$ and $e(F) \geq e(K^{(d+1)}_{3,3,\ldots,3})$ for all $F \in \mathcal{F}_d$.
The $(d+1)$-graph $K^{(d+1)}_{3,3,\ldots,3}$ has $3\cdot(d+1)$ vertices and $3^{d+1}$ edges.
Thus by Proposition \ref{lowerBoundEx(n,F)} there is $C>0$ and $N \in \mathbb{N}$ such that
\[
 \ex(n,\mathcal{F}_d) \geq C \cdot n^{d+1-\frac{3(d+1)-(d+1)}{3^{d+1}-1}} \text{ for } n \geq N,
\]
which is what we had to show.
\end{proof}
To conclude the part on upper bounds let us summarize our results.
For all~$d$ and for large $n$ we now have:
\[
 C \cdot n^{d+1-\frac{2(d+1)}{3^{d+1}-1}} \leq \ex(n,\mathcal{F}_d) < n^{d+1-\frac{1}{3^d}},
\]
which shows that improving the upper bound on, or even determining $\ex(n,\mathcal{F}_d)$ could not close the gap between the upper bound of $O(n^{d+1-\frac{1}{3^d}})$ and the lower bound of $\Omega(n^d)$ for 
$ \max\left\{f_d(K)\;\middle|\; \dim(K)=d, |V(K)|=n, K \hookrightarrow \R^{2d}\right\}$ by much.
For $d=2$ we get
\[
 C \cdot n^{3-\frac{3}{13}} \leq \ex(n,\mathcal{F}_2) < n^{3-\frac{1}{9}}
\]
for large enough $n$.
Observe that for the case $d=1$ we proved the existence of a graph containing neither $K_5$ nor $K_{3,3}$ with at least $C \cdot n^{\frac{3}{2}}$ edges for \mbox{large $n$}.
However, simple explicit constructions for such graphs are already known. One example is the following:
Since both $K_5$ and $K_{3,3}$ contain a quadrilateral $C_4$, quadrilateral-free graphs are examples of $K_5$- and $K_{3,3}$-free graphs. It is known that for any prime power $q$ there is a graph with $n=q^2+q+1$ vertices and $\frac{1}{2}q^2(q+1)=\Theta(n^{\frac{3}{2}})$ edges, defined by the points and lines of the projective plane $PG(q,2)$ over the field $\mathbb{F}_q$, that doesn't contain a quadrilateral (\cite{Erdos.1966}, \cite{Brown.1966}).
\chapter*{Further Thoughts}
\addcontentsline{toc}{chapter}{Further Thoughts}
Let us summarize our results.
In Problem \ref{dieFrage} we asked for
\[
 \max\left\{f_d(K) \;\middle|\; \dim(K)=d,\, |V(K)|=n,\, ||K|| \hookrightarrow \R^{\dR} \right\}
\]
for fixed $n \geq d+1$ and $d,\dR \geq 1$ such that $d \leq \dR \leq 2d$.

We could achieve a lower bound of $f_d(C_{\dR + 1}(n)) = \Omega(n^{\lceil\frac{\dR}{2}\rceil})$ for all $d,\dR \geq 1$ such that $d \leq \dR \leq 2d$ in Chapter \ref{lowerBounds}.
For $\dR=2d$ we could also prove that this bound can not be improved by simply adding simplices.

In Chapter \ref{upperBounds} we could improve the trivial upper bound of $\ueber{n}{d+1}$ by giving an upper bound of \[\ex(n,\mathcal{F}_d) = O(n^{d+1-\frac{1}{3^d}}).\]
We could also see that a better estimate of $\ex(n,\mathcal{F}_d)$ would not improve this bound considerably by showing that \[\ex(n,\mathcal{F}_d) \geq C \cdot n^{d+1-\frac{2(d+1)}{3^{d+1}-1}}.\]

To conclude, we now collect some ideas that could be pursued to get better results for Problem \ref{dieFrage}.
To get Corollary \ref{lowerBoundExEffDee} we proved the existence of an $\mathcal{F}_d$-free $(d+1)$-graph on $n$ vertices with at least \[C \cdot n^{d+1-\frac{2(d+1)}{3^{d+1}-1}}\] edges for large $n$.
For the case $d=1$ this shows that for large $n$ there is a graph not containing $K_5$ and $K_{3,3}$ with at least $C \cdot n^{\frac{3}{2}}$ edges.
By Wagner's Theorem (Theorem \ref{Wagner}) these two graphs characterize planar graphs as excluded minors; Kuratowski's Theorem (Theorem \ref{Kuratowski}) states that excluding subdivisions of $K_5$ and $K_{3,3}$ gives all planar graphs.
As the maximum size of a planar graph is of linear order (Lemma \ref{sizePlanarGraphs}), this shows that in this case the order of the maximal number of edges changes if one excludes subdivisions or minors instead of subgraphs.
%As this could also be true for $k$-graphs with $k\geq2$, trying to exclude minors in these cases might be interesting.

This gives rise to hopes that excluding subdivisions of the hypergraphs in $\mathcal{F}_d$ or excluding themselves as minors might lead to a better upper bound than the one achieved in Proposition \ref{theUpperBound}.
However, while the concepts of subdivisions and minors are well understood for graphs, they appear to be not as clear and are less studied in the context of hypergraphs. In addition to that, the connection to embeddability is not as obvious. Nevo's concept of minors for simplicial complexes, which was used in this thesis, (Definition \ref{definitionMinor}) doesn't seem to translate easily to hypergraphs.

Another idea which doesn't seem easy to rule out is to use the Upper Bound Theorem for spheres in the following way:
If a $d$-complex $K$ on $n$ vertices is a subcomplex of some simplicial $r$-sphere $S$ with $n'=O(n)$ vertices, we know that $f_d(K) \leq f_d(S) \leq f_d(C_{r+1}(n')) = O(n'^{\lfloor\frac{r}{2}\rfloor})= O(n^{\lfloor\frac{r}{2}\rfloor})$ by the Upper Bound Theorem for spheres.
Suppose we would know that every $d$-complex $K$ on $n$ vertices that embeds into $\R^r$ is a subcomplex of such a sphere. This seems unlikely, but there also seems to be no counterexample yet.
If it were true, we would have an upper bound of $O(n^{\lfloor\frac{r}{2}\rfloor})$ for Problem \ref{dieFrage}, which would at least be of the same order as the one in Conjecture \ref{Conj3}. 

Furthermore, it might be possible to show that it suffices to prove Conjecture~\ref{Conj3} for all even $\dR$. The approach to this draws on the idea of the proof of Lemma~\ref{zweiDrei}, the solution for the case $d=2$ and $\dR=3$, which depends on the solution for the case $d=1$ and $\dR=2$ (planar graphs).

In this proof, the lower dimensional case gives upper bounds for the $f$-vectors of the vertex links of the complex. These bounds are then, via double counting, used to determine an upper bound for the whole complex.
An adaptation of this proof for higher \emph{odd} dimensions seems possible. For odd $\dR=2k+1$, we have $f_d(C_{\dR+1}(n))=O(n^{\lceil\frac{\dR}{2}\rceil})=O(n^{k+1})$, while $f_{d-1}(C_{\dR}(n))=O(n^{\lceil\frac{\dR-1}{2}\rceil})=O(n^k)$. Thus, the additional factor of $n$, which arises by summing over all vertices, conserves the right order in the odd case, whereas this inductive idea clearly fails for even $\dR$.
This idea might also work for reasonable upper bounds in even dimensions, other than $f_d(C_{\dR+1}(n))$.
%This might also be a way to obtain reasonable upper bounds in odd dimensions from ones in even dimension.
%
\newpage
\addcontentsline{toc}{chapter}{Bibliography}
\bibliographystyle{amsalpha} %...oder doch lieber ?amsplain
\bibliography{Literatur}
\chapter*{Zusammenfassung}
\thispagestyle{empty}
\setlength{\parindent}{0pt}
Es ist nicht schwer zu zeigen, dass jeder $d$-dimensionale Simplizialkomplex eine Einbettung in den $\R^{2d+1}$ besitzt. Demzufolge ist die maximale Anzahl von \mbox{$d$-Simplizes} f\"ur einen in diesen Raum einbettbaren Komplex, der $n$ Ecken besitzt,  $\ueber{n}{d+1}= \Theta(n^{d+1})$.

F\"ur dem Fall $d=2$ liefert dies $\Theta(n^3)$ f\"ur Einbettbarkeit in den $\R^5$. Mit elementaren Methoden kann auch gezeigt werden, dass ein $2$-dimensionaler Simplizialkomplex auf $n$ Ecken, der in den $\R^3$ eingebettet werden kann, h\"ochstens aus $n(n-3) = \Theta(n^2)$ Dreiecken besteht. F\"ur in den $\R^4$ eingebettete \mbox{$2$-Komplexe} ist die Frage der maximalen Anzahl an Dreiecken ungekl\"art.

Diese Arbeit befasst sich mit der allgemeineren Frage, wieviele maximale Simplizes ein $d$-dimensionaler Simplizialkomplex, der sich in den $\R^{\dR}$, f\"ur $d \leq \dR \leq 2d$, einbetten l\"asst, h\"ochstens enthalten kann.
Mit Hilfe von zyklischen Polytopen, deren Randkomplexe Beispiele einbettbarer Komplexe bilden, erh\"alt man die untere Schranke $f_d(C_{\dR + 1}(n)) = \Omega(n^{\lceil\frac{\dR}{2}\rceil})$.

Es wird gezeigt, dass das Hinzuf\"ugen eines weiteren Simplizes zu dem Komplex $\mathcal{C}(\partial C_{\dR + 1}(n))^{\leq d}$ bei Verwendung der vorhandenen Ecken zu einem nicht mehr in den $\R^{\dR}$ einbettbaren Komplex f\"uhrt.
Die Schranke kann also nicht auf diesem einfachen Wege verbessert werden.

Um eine obere Schranke zu bekommen, wird hier die Idee auszuschlie\ss{}ender Unterkomplexe verfolgt.
Eine Verallgemeinerung des Satzes von van Kampen und Flores liefert eine Klasse von nicht in den $\R^{\dR}$ einbettbarer Komplexen. Ein Simplizialkomplex, der einen solchen als Unterkomplex besitzt, kann ebenfalls nicht einbettbar sein. Demzufolge enth\"alt die Klasse der Komplexe, die kein Oberkomplex eines dieser nicht einbettbaren Komplexe sind, die Klasse der in den $\R^{\dR}$ einbettbaren Komplexe.

F\"ur den Fall $\dR=2d$ kann mit den Methoden der extremalen Hypergraphentheorie eine obere Schranke f\"ur die maximale Anzahl von Simplizes in einem Simplizialkomplex, in dem die gegebenen nicht einbettbaren Komplexe als Unterkomplexe verboten sind, gefunden werden.
Diese ist also auch eine obere Schranke f\"ur die eigentliche Fragestellung im Fall $\dR=2d$ und hat die Gr\"o\ss{}en\-ordnung $O(n^{d+1-\frac{1}{3^d}})$.
Desweiteren wird gezeigt, dass man auf diesem Wege keine bessere Absch\"atzung als $O(n^{d+1-\frac{2(d+1)}{3^{d+1}-1}})$ erhalten kann.
\end{document}